\documentclass[11pt]{amsart}

\usepackage{bm}
\usepackage{hyperref}
\usepackage{setspace}
\usepackage{fullpage}
\usepackage{braket}
\usepackage{verbatim}
\usepackage{lscape}
\usepackage{tabularx}
\usepackage{tikz-cd}
\usepackage{amsmath}
\usepackage{pgfplots}
\usepackage{enumitem}
\usepackage{amsthm,thmtools}
\usepackage{extarrows}
\usepackage{xypic}
\usepackage{ulem}
\usepackage{graphicx}
\usepackage{float}
\usepackage{graphics,amssymb}
\usepackage{amsfonts}
\usepackage{latexsym}
\usepackage{epsf}
\usepackage{mathrsfs}
\usepackage{bbm}
\usepackage{hyperref}
\usepackage{color}
\usepackage[framemethod=tikz]{mdframed}
\usepackage{cancel}
\usepackage{tikz}
\usetikzlibrary{positioning,calc,intersections,through,backgrounds,matrix,arrows}
\usepackage{mathtools}
\usepackage{url}
\usepackage{fancybox}
\usepackage{dsfont}
\usepackage{bbding}
\usepackage{booktabs}

\usepackage[toc,page]{appendix}
\usepackage{setspace}
\usetikzlibrary{fit}
\usepackage[low-sup]{subdepth}

\newtheorem{theorem}{Theorem}[section]

\newtheorem*{Corollary*}{Corollary}

\newtheorem{lemma}[theorem]{Lemma}
\newtheorem{proposition}[theorem]{Proposition}

\newtheorem{claim}[theorem]{Claim}

\theoremstyle{definition}
\newtheorem{remark}{Remark}[section]

\newtheorem{definition}[theorem]{Definition}

\newtheorem{setup}[theorem]{Set-up}

\newcommand{\Gal}[2]{\mathrm{Gal}(#1/#2)}

\newcommand{\C}{\mathbb{C}}
\newcommand{\p}{\mathfrak{p}}

\newcommand{\Z}{\mathbb{Z}}
\newcommand{\Q}{\mathbb{Q}}

\newcommand{\val}{\mathrm{val}}
\newcommand{\Hom}{\mathrm{Hom}}

\newcommand{\Res}{\mathrm{Res}}
\newcommand{\cyc}{\epsilon}

\newcommand{\AAA}{\mathbb{A}}
\newcommand{\Frob}{\mathrm{Frob}}
\newcommand{\Rmin}{\mathrm{R}^{\min}}

\newcommand{\OO}{\mathrm{O}}

\newcommand{\Gl}{\mathrm{GL}}

\newcommand{\ord}{\mathrm{ord}}

\newcommand{\TT}{\mathcal{T}}

\newcommand{\Iw}{\mathrm{Iw}}

\newcommand{\Ind}{\mathrm{Ind}}

\newcommand{\tr}{\mathrm{Tr}}

\newcommand{\pcT}{\mathrm{T}}

    \DeclareFontFamily{U}{wncy}{}
    \DeclareFontShape{U}{wncy}{m}{n}{<->wncyr10}{}
    \DeclareSymbolFont{mcy}{U}{wncy}{m}{n}
    \DeclareMathSymbol{\Sha}{\mathord}{mcy}{"58}

\usepackage{stmaryrd} 
\usepackage{todonotes}

\newcommand{\GSp}{\mathrm{GSp}}

\newtheorem*{theorem*}{Theorem}
\date{}

\title{A minimal modularity lifting theorem for Siegel modular forms}
\author{Shaunak V. Deo}
\address{Department of Mathematics, Indian Institute of Science, CV Raman Road, Bengaluru 560012, India}
\email{shaunakdeo@iisc.ac.in}
\author{Bharathwaj Palvannan}
\address{Department of Mathematics, Indian Institute of Science, CV Raman Road, Bengaluru 560012, India}
\email{bharathwaj@iisc.ac.in}


\usepackage{amsrefs}

\usepackage[T1]{fontenc} % Font package
\usepackage{libertine}
\usepackage[utf8]{inputenc}
\usepackage{arydshln}

\begin{document}

\begin{abstract}
We prove a minimal modularity lifting theorem (in the spirit of Genestier--Tilouine and Pilloni) in the setting of Siegel modular forms of genus two when the residual representation arises from a stable Yoshida lift, that is, an automorphic induction of a nearly ordinary Hilbert modular eigencuspform over a real quadratic field. As applications of the underlying $R=\mathbb{T}$ theorem, we establish the freeness of a universal minimal ordinary Galois deformation ring over an Iwasawa algebra in two variables along with the uniqueness of Hida families passing through classical $p$-ordinary Siegel modular eigenforms with \textit{very regular} weights. 
\end{abstract}

\maketitle

\section{Introduction}
\label{sec:intro}
Modularity lifting theorems played a key role in Wiles' celebrated work on Fermat's last theorem \cite{MR1333035}. They are consequences of ``$R=\mathbb{T}$'' theorems which assert that a universal deformation ring parametrizing certain lifts of a mod $p$ Galois representation is isomorphic to a Hecke algebra acting on a suitable space of modular forms (or more generally, automorphic forms).
The ``$R=\mathbb{T}$'' theorem used in Wiles' work is established using the \textit{Taylor--Wiles method} \cite{MR1333036}. Since then proving modularity lifting theorems in various settings and studying their arithmetic applications have become an important theme in number theory. These theorems are mostly obtained using the Taylor--Wiles method (and its subsequent extensions). The main purpose of this paper is to establish a \textit{minimal} modularity lifting theorem in the setting of Siegel modular forms of genus $2$ using the Taylor--Wiles method when the residual representation arises from a stable Yoshida lift. Our methods closely follow the works of Genestier--Tilouine \cite{MR2234862} and Pilloni \cite{MR2920881} who establish similar theorems for Siegel modular forms of genus $2$ using the Taylor--Wiles method; see \S \ref{subsec:comparisongtp} for a comparison with earlier works. As a by-product of our main theorem, we establish the freeness of a universal minimal ordinary Galois deformation ring over an Iwasawa algebra in two variables, along with the uniqueness of $\GSp_4$ Hida families passing through certain Siegel modular forms with \textit{very regular weights}. Our $R=\mathbb{T}$ theorem plays an important role in establishing the results in a companion work \cite{DP} where we obtain local characterizations of \textit{stable Yoshida components} of Hida families of Siegel modular forms in the spirit of Ghate--Vatsal \cite{MR2139691}. We now describe our set-up and results.\\

Let $p\geq 5$ denote a prime. Let $K$ be a real quadratic field where the prime $p$ splits as $\p_1\p_2$. Let $O_K$ denote its ring of integers. We fix embeddings  $\iota : \overline{\Q} \hookrightarrow \C$ and $\iota_p : \overline{\Q} \hookrightarrow \overline{\Q_p}$. There are two embeddings $\sigma_1, \sigma_2: K \hookrightarrow \overline{\Q}$. The prime $\p_1$ is determined by $\iota_p \circ \sigma_1$. We choose the ordering of the weights of Hilbert modular forms over $K$ with respect to $\iota \circ \sigma_1$ and $\iota \circ \sigma_2$ respectively. Let $\mathfrak{M}$ be an ideal of $O_K$ co-prime to $p$. Let $S$ denote the finite set of primes of $\Q$ consisting of the primes dividing the norm of $\mathfrak{M}$, the primes ramified in the extension $K/\Q$ along with $p$ and $\infty$. Let $\Q_S$ denote the maximal extension of $\Q$ unramified outside of $S$. Let $G_{\Q,S}$ denote $\Gal{\Q_S}{\Q}$. By abuse of notation, we continue to denote by $S$ the set of primes of $K$ lying above $S$. Let $G_{K,S}$ denote the subgroup $\Gal{\Q_S}{K}.$ \\

 Let $g_0$ be a $p$-nearly-ordinary cuspidal Hilbert modular newform over $K$, with regular weight $(\kappa_1,\kappa_2)$, trivial nebentypus and level $\mathfrak{M}$. We further assume that $g_0$ is nearly ordinary at $\p_1$ and $\p_2$ (see Definition~\ref{def:ordgl2}). We suppose that  $\kappa_1 \equiv \kappa_2 \ (\mathrm{mod} \ 2)$ with $ \kappa_1 \geq\kappa_2 \geq 2$. We let $\tau:G_{K,S} \rightarrow \Gl_2(\overline{\Q}_p)$ denote the $p$-adic Galois representation associated to $g_0$. Denote by $\tau^c$, the conjugate Galois representation of $\tau$ coming from a non-trivial outer automorphism of $G_{K,S}$ induced by $\Gal{K}{\Q}$.  

\begin{setup}\label{setting:all} Let $\mathbb{F}$ be a finite field with characteristic $p$ over which the residual representation $\bar\tau:G_{K,S} \rightarrow \Gl_2(\mathbb{F})$ associated to $g_0$ is defined. Along with those stated earlier, we impose the following hypotheses throughout the paper. These are standard hypotheses that are usually imposed to make the minimal ordinary deformation problem tractable. We give the rationale behind some of these hypotheses in \S \ref{subsec:rationale}.

 \begin{enumerate}[style=sameline, align=left,label=(\scshape{Ind}), ref=\scshape{Ind},partopsep=0pt,parsep=0pt]
  \item\label{lab:ind} The residual representation $\bar\tau$  is absolutely irreducible. Furthermore, the induced representation $\Ind^\Q_K(\bar\tau)$, which we denote by $\bar\rho$, is an absolutely irreducible residual $G_{\Q,S}$-representation.
\end{enumerate}

\begin{enumerate}[style=sameline, align=left,label=(\scshape{$p$-dist}), ref=\scshape{$p$-dist},partopsep=0pt,parsep=0pt]
  \item\label{lab:pdist}  The four characters appearing in the semi-simplification of $\bar{\rho} \mid_{G_{\Q_p}}$ are distinct. 
\end{enumerate}

\begin{enumerate}[style=sameline, align=left,label=(\scshape{optimallevel}), ref=(\scshape{optimallevel}),partopsep=0pt,parsep=0pt]
    \item\label{hyp:optimallevel} The (tame) Artin conductor of $\bar\tau$ equals $\mathfrak{M}$.
\end{enumerate}

\begin{enumerate}[style=sameline, align=left,label=(\scshape{squarefree}), ref=(\scshape{squarefree}),partopsep=0pt,parsep=0pt]
    \item\label{hyp:squarefree} The ideal $\mathfrak{M}$ is generated by a squarefree integer $M$ such that $(M,p\Delta_K)=1$, where $\Delta_K$ is the discriminant of $K$.
\end{enumerate}

\begin{enumerate}[style=sameline, style=sameline, align=left,label=(\scshape{cong}), ref=(\scshape{cong}),partopsep=0pt,parsep=0pt]
    \item\label{hyp:congruence1} If $\ell$ divides $\mathfrak{M}$, then $p$ does not divide $(\ell^4-1)$. 
\end{enumerate}
\begin{enumerate}[style=sameline, style=sameline, align=left,label=(\scshape{cong-ram}), ref=(\scshape{cong-ram}),partopsep=0pt,parsep=0pt]
    \item\label{hyp:congruence2}  If $\ell$ ramifies in $K$  and $\lambda_{\ell}$ denotes the unique prime of $K$ lying above $\ell$, then the natural image of $(\ell^4-1)(a_{\lambda_\ell}(g_0)^2-(\ell+1)^2\ell^{\kappa_1-2})$ is not $0$ in $\overline{\mathbb{F}}_p$. Here, $a_{\lambda_\ell}(g_0)$ is the $T_{\lambda_\ell}$ eigenvalue of $g_0$.
\end{enumerate}

\begin{enumerate}[style=sameline, style=sameline, align=left,label=(\scshape{bigimage}), ref=(\scshape{bigimage}),partopsep=0pt,parsep=0pt]
   \item\label{hyp:bigimage} $\mathrm{SL}_2(\mathbb{F}_p) \times \mathrm{SL}_2(\mathbb{F}_p) \subset \text{Im}(\bar\tau \oplus \bar\tau^c)$.
\end{enumerate}

\end{setup}

\subsection{Stable Yoshida lifts and Hida theory for $\GSp_4$}

Let $f_0$ denote the stable Yoshida lift of $g_0$, defined over $\Q$. Note that under our assumptions, $f_0$ turns out to be a $p$-ordinary cuspidal Siegel eigenform of weight $(k_1,k_2)$ with level $\Gamma_0^{(2)}(M\Delta_K)$ (see \cite{MR3858470}) where
\begin{align} \label{formula:weights}
    k_1 = \dfrac{\kappa_1+\kappa_2}{2} \geq  k_2= \dfrac{\kappa_1-\kappa_2}{2}+2.
\end{align}

 Let $\TT$ be the local component of the ordinary $\GSp_4$-Hecke algebra passing through $f_0$ as considered by Hida \cite{MR1954939} and Pilloni \cite{MR3059119} with tame level $\Gamma_0^{(2)}(M\Delta_K)$. This Hecke algebra is a local ring and is generated by the Hecke operators $T_{\ell,0}$, $T_{\ell,1}$ and $T_{\ell,2}$ for all primes $\ell \nmid pM\Delta_K$, along with the Hecke operators $U_{p,1}$ and $U_{p,2}$. There exists a canonical subring $\Z_p\llbracket y_1,y_2\rrbracket$ --- call it $\Lambda$ henceforth --- of $\TT$, where $y_1$ and $y_2$ denote the \textit{weight} variables such that $\TT$ is a finite integral extension of $\Z_p\llbracket y_1,y_2\rrbracket$.

\subsection{Main Results}
\label{subsec:mainresults}

Let $\Rmin$ denote the universal \textit{minimal} ordinary Galois deformation ring, parametrizing minimal ordinary $\GSp_4$-valued lifts of $\bar{\rho}$. See Definition~\ref{def:orddef} for its precise definition. Our main result is an $\mathrm{R}=\mathbb{T}$ theorem, where we apply the Taylor-Wiles method \cite{MR1333036,MR1440309,fujiwara2006deformation}, following the approaches of Genestier--Tilouine \cite{MR2234862},  Tilouine \cite{MR2264659} and Pilloni \cite{MR2920881}. 

\begin{restatable}{Theorem}{theoremone}\label{thm:R=T} 
\label{thm::rist}
Under Set-up \ref{setting:all}, the natural surjective morphism $\Rmin \stackrel{\simeq}{\to} \TT$ is an isomorphism. Furthermore,  $\TT$ is a  free $\Lambda$-module with finite rank.
\end{restatable}

As an immediate consequence of Theorem~\ref{thm::rist} and the vertical control theorem of Hida and Pilloni (Theorem~\ref{thm:verticalcontrol}), we get the following minimal modularity result alluded to in the title:

\begin{Corollary*}
    Suppose we are in Set-up \ref{setting:all}. Let $k'_1$ and $k'_2$ be positive integers such that $k'_i \equiv k_i \pmod{(p-1)}$ for $i=1,2$ and $k'_1 \geq k'_2 \geq 4$. Suppose $\rho : G_{\Q,S} \to \GSp_4(\overline{\Z}_p)$ is a minimal deformation of $\bar\rho$ of weight $(k'_1,k'_2)$, in the sense of Definition~\ref{defi:Qdef}. Then there exists a $p$-ordinary Siegel modular eigenform of genus $2$ with level $\Gamma_0^{(2)}(M\Delta_K)$ and weight $(k'_1,k'_2)$ such that $\rho$ is the $p$-adic Galois representation attached to it.
\end{Corollary*}

As in \cite{MR3059119}, we say that a classical Siegel modular form of genus $2$ has \textit{very regular weight} $(k'_1,k'_2)$ if $k'_1>k'_2\gg_{p,M,\Delta_K} 0$.  As another application of Theorem \ref{thm::rist}, we prove \textit{uniqueness} of the Hida family passing through $p$-ordinary Siegel modular forms with very regular weight. 

\begin{restatable}{Theorem}{theoremtwo}  \label{thm:yoshidafamily}
Suppose we are in Set-up \ref{setting:all}. Suppose $k'_1$ and $k'_2$ are integers such that $k'_i \equiv k_i \pmod{(p-1)}$. Suppose also that $f$ is a classical specialization of $\TT$ with weight $(k'_1,k'_2)$ that is very regular. Then, there exists a unique Hida family passing through $f$. 
\end{restatable}

\begin{remark}\label{rem:nocontrolhecke}
The freeness of the ordinary Hecke algebra over $\Lambda$ cannot be automatically deduced (in general) in the $\GSp_4$ setting from the vertical control theorems of Hida \cite{MR1954939} and Pilloni \cite{MR3059119} for Siegel modular forms due to the absence of duality between Siegel modular forms and Hecke algebra. If we let $P_{(k_1,k_2)}$ denote a classical height two prime ideal in $\Lambda$, then the control theorems for Siegel modular forms only let us deduce that the natural surjective ring homomorphism 
\begin{align} \label{eq:sp}
\mathrm{sp}_{(k_1,k_2)}: \TT \otimes_{\Lambda} \dfrac{\Lambda}{P_{(k_1,k_2)}} \twoheadrightarrow \mathcal{H}_{(k_1,k_2)},
\end{align}
has nilpotent kernel. Here, $\mathcal{H}_{(k_1,k_2)}$ denotes the corresponding localized Hecke algebra at finite weight $(k_1,k_2)$. See \cite[Theorem 6.16]{MR3148103} and \cite[Proposition 2.3]{hsieh_yoshida}. For these reasons, we need to appeal to ``$\mathrm{R}=\mathbb{T}$'' theorems to deduce freeness of $\TT$ over $\Lambda$. \\

Similarly one needs to appeal to $R=\mathbb{T}$ theorems at fixed weights to deduce uniqueness of Hida families in  Theorem \ref{thm:yoshidafamily}. An important ingredient towards establishing these $R=\mathbb{T}$ theorems is Pilloni's horizontal control theorem \cite{MR3059119} which is established for very regular weights. For this reason, we obtain the uniqueness result of Theorem \ref{thm:yoshidafamily} only for Siegel modular forms of weights $(k_1',k_2')$ with $k_1' > k_2' \gg 0$.
\end{remark}

\subsection{Comparison with earlier works} \label{subsec:comparisongtp}
Our approach towards proving Theorem \ref{thm::rist} follows Pilloni's strategy  in  \cite{MR2920881}, wherein he works with the deformation problem formulated by Genestier--Tilouine \cite[Section 4]{MR2234862}, verifies certain hypotheses (labeled H1, H2 and H3) and uses his horizontal control theorems \cite[Theorem 9.1]{MR3059119} to apply the refinement of Taylor--Wiles method found by Diamond \cite{MR1440309} and Fujiwara \cite{fujiwara2006deformation}. Nonetheless, we highlight some key differences from earlier works such as \cite{MR2234862}, \cite{MR2264659} and \cite{MR2920881}.  \\

Firstly, the theorems of Genestier--Tilouine \cite{MR2234862} and Tilouine \cite{MR2264659} involve the Hecke algebra $\TT_{(3)}$ acting on the $\Lambda$-adic module obtained from $H^3$ (Betti cohomology)  whereas we follow Hida \cite{MR1954939} and Pilloni \cite{MR3059119} in considering the Hecke algebra $\TT$ acting on the $\Lambda$-adic module obtained from $H^0$ (coherent cohomology). Although Eichler--Shimura type morphisms would provide us a ring map $\TT_{(3)} \rightarrow \TT$, we feel it is a delicate question to show that the natural map between these $\Lambda$-adic Hecke algebras is an isomorphism (without having established that $\TT$ itself is also free over $\Lambda$). For these reasons, we are also unable to use strong results of Mokrane--Tilouine \cite{MR1944174} for deducing freeness of $\TT$ over $\Lambda$. \\

Secondly, we rely on the explicit constructions of stable Yoshida lifts coming from the work of Hsieh--Namikawa \cite{MR3623733,MR3858470}. This forces our tame level to be a Siegel congruence subgroup.   On the other hand, Pilloni \cite{MR2920881} considers mostly squarefree levels coming from Klingen subgroups since the shape of the inertial Galois representation at the primes $\ell$ dividing the level are then better controlled. For instance, see \cite{MR2234862,MR2344630,MR2735984}.  For these reasons, the hypotheses imposed in Set-up \ref{setting:all} enable explicit computations as shown in \S \ref{sec:galoisrep}, to directly deduce various properties of the global and local Galois representations, thereby showing that the natural Galois representation valued in $\TT$ does in fact satisfy the  minimal deformation properties. These calculations are not present in the works of Genestier--Tilouine \cite{MR2234862} and Pilloni \cite{MR2920881} as they use local Langlands correspondence (LLC) for $\GSp_4$ to directly establish these properties. It also seems that the paramodular subgroups \cite{MR2344630} of $\mathrm{Sp}_4(\Z)$ have a closer connection to the Artin conductor of the $p$-adic Galois representation attached to Siegel modular forms.

\subsection{The rationale behind our hypotheses in Set-up \ref{setting:all}} \label{subsec:rationale}

The hypotheses \ref{hyp:squarefree}, \ref{hyp:congruence1} and \ref{hyp:congruence2} are used to determine the shape of the local Galois representation at primes $\ell \in S \setminus \{p\}$. These hypotheses are imposed since in our setting we are unable to glean these shapes only using LLC for $\GSp_4$ \cite{MR2800725} and results of Brooks--Schmidt \cite{MR2344630} which is in contrast with the situation in \cite{MR2234862} and \cite{MR2920881}. \\

We make the following observation about the validity of the hypothesis \ref{hyp:congruence2}.   Since $g_0$ has regular weight, it follows from the work of Carayol (\cite{MR870690}) and Blasius (\cite{MR2327298}) that the Ramanujan conjecture holds for $g_0$. See  \cite{MR2327298}, especially the remarks on page 15. Using the Ramanujan conjecture, hypothesis \ref{hyp:squarefree} along with the fact that $g_0$ has trivial nebentypus allows us to make the following observation:  if $\lambda'$ is a prime of $\mathcal{O}_K$ co-prime to $\mathfrak{M}$ and $\ell'$ is the cardinality of its residue field, then the quantity \begin{align}\label{ref:rameqn}
    a_{\lambda'}(g_0)^2 - (\ell'+1)^2(\ell')^{\kappa_1-2}
\end{align}
is non-zero. This is because if the quantity in equation (\ref{ref:rameqn}) equals zero, then as the nebentypus is trivial, the  roots of the Hecke polynomial at $\lambda'$ equal $(\ell')^{\kappa_1/2}$ and $(\ell')^{\kappa_1/2-1}$. This would force the local component at $\lambda'$ of the unitary normalization of the  automorphic representation, corresponding to $g_0$, to be a complementary series (which is forbidden by the Ramanujan conjecture). Thus, the hypothesis~\ref{hyp:congruence2} holds for all but finitely many primes $p$. However, note that it is still an open question that there are infinitely many ordinary primes for a general non-CM Hilbert modular form. \\

 The hypothesis \ref{hyp:bigimage} is required in the proof of Theorem~\ref{thm::rist} in two essential ways. Firstly, it allows us to use the Taylor-Wiles method. Secondly, it allows us to prove in Theorem \ref{thm:symplecticbasis} that the lift of $\bar\rho$ to $\TT$ is valued in $\GSp_4(\TT)$. During the course of the proof of  Theorem \ref{thm:symplecticbasis}, we verify that $\text{Im}(\bar\rho)$ satisfies a property very similar to the \textit{tr\'{e}s symplectique} property formulated by Pilloni (\cite[Definition 3.3]{MR2920881}). Note that since $\bar\tau$ is the residual representation of $g_0$, when $g_0$ is a non CM form, it follows from \cite[Proposition 3.8, Corollary 3.14]{MR2172950} that the hypothesis \ref{hyp:bigimage} holds for all but finitely many primes $p$.

\subsection*{Acknowledgements}

We thank Frank Calegari, Mladen Dimitrov, David Loeffler, Haruzo Hida, Ming-Lun Hsieh, Vincent Pilloni, Brooks Roberts, Ralf Schmidt and Jacques Tilouine for helpful correspondences and discussions that clarified many of our questions. Both SD and BP are partially supported by the Infosys Young Investigator Awards from the Infosys Foundation Bangalore, DST FIST program 2021 [TPN - 700661]  along with the SERB-MATRICS grants MTR/2023/000850 and MTR/2022/000244 respectively. 

\tableofcontents
\section{Siegel modular forms over $\Q$ with genus two} \label{sec:siegelgenus2}

In this section, we gather the required background information on Hida families associated to Siegel modular forms over $\Q$ in the genus two setting along with Galois theoretic properties of stable Yoshida lifts (automorphic inductions from Hilbert modular forms over real quadratic fields). 

Let $R$ be a commutative ring.  Recall that the general symplectic group $\GSp_4(R)$ is defined below:
\begin{align*}
\GSp_4(R) := \left\{g \in \Gl_4(R), \text{ such that there exists a $\lambda(g) \in R^\times$ satisfying } g^T \underbrace{\left[\begin{array}{cccc} 0 & 0 & 0 & 1 \\ 0 & 0 & 1 & 0 \\ 0 & - 1 & 0 & 0 \\ -1 & 0 & 0 & 0 \\ \end{array}\right]}_Jg = \lambda(g)J \right\}.
\end{align*}
The map $\lambda: \GSp_4(R) \rightarrow R^\times$ turns out to be a character, called the similitude character. \\

Let $\varPi$ be an irreducible cuspidal automorphic representation of $\GSp_4(\AAA_\Q)$. We refer the reader to \cite[Section 5]{MR546598} and \cite[Section 4.2]{MR3748235} for the definition of an irreducible cuspidal automorphic representation of $\GSp_4(\AAA_\Q)$. We shall also consider its restricted tensor product decomposition $\varPi \cong \varPi_\infty \otimes \left(\sideset{}{^\prime}\bigotimes_{\ell \text{ finite}} \varPi_\ell\right)$ involving irreducible admissible representations of $\GSp_4(\mathbb{R})$ and $\GSp_4(\Q_\ell)$ (see work of Flath \cite{MR546596}). The representation $\varPi_\infty$ is called the local component of $\varPi$ at $\infty$. Throughout this section, we shall impose the following hypotheses on $\varPi_\infty$ and the \textit{level} of $\varPi$ respectively:
\begin{enumerate}
\item  $\varPi_\infty$ is a holomorphic (limit of) discrete series representation (in the sense of \cite[Section 2.5]{MR3748235}) with \textit{regular} weight $(k'_1,k'_2)$ with $k'_1 \geq k'_2 \geq 2$.
\item  $\varPi$ has vectors that are right-invariant under the action of the adelization of a principal Siegel congruence subgroup $\Gamma(N')$ for some $N' \geq 3$ that is co-prime to~$p$.
\end{enumerate}

We let $S_{\mathrm{good}}$ denote the set of finite primes where $\varPi_\ell$ is unramified. We let $Q_{\varPi_\ell}(X)$ denote the monic degree four Hecke polynomial at $\ell$:
\begin{align*} 
Q_{\varPi_\ell}(X):=X^4 - T_{\ell,2}(\varPi) X^3  + \ell(T_{\ell,1}(\varPi)+(\ell^2+1)T_{\ell,0}(\varPi))X^2 - \ell^3T_{\ell,2}(\varPi) T_{\ell,0}(\varPi) X + \ell^6T_{\ell,0}^2(\varPi).
\end{align*} 
If $\ell$ belongs to $S_{\mathrm{good}}$, then $\varPi_\ell$ has an action of the $\Z$-algebra consisting of all $\GSp_4(\Z_\ell)$ bi-invariant continuous functions $\GSp_4(\Q_\ell) \rightarrow \overline{\Q}$ with compact support (with multiplication given by convolution). This algebra is often called the spherical Hecke algebra and denoted $\mathcal{H}(\GSp_4(\Q_\ell),\Z)$. The spherical Hecke algebra turns out (see \cite[Section 5.1.3]{MR4105535}) to be isomorphic to $\Z[T_{\ell,0},\ T_{\ell,0}^{-1}, \ T_{\ell,1}, \ T_{\ell,2} ]$.
Here, $T_{\ell,2}$, $T_{\ell,1}$ and $T_{\ell,0}$ are (respectively) the characteristic functions of the following double cosets:

\begin{align*}
\GSp_4(\Z_\ell)\left[\begin{array}{cccc} 1 \\ & 1 \\ & & \ell \\ & & & \ell \end{array} \right] \GSp_4(\Z_\ell),  \quad  \GSp_4(\Z_\ell)\left[\begin{array}{cccc} 1 \\ & \ell \\ & & \ell \\ & & & \ell^2 \end{array} \right] \GSp_4(\Z_\ell),  \quad  \ell \GSp_4(\Z_\ell).
\end{align*}

\begin{definition}\label{def:gsppordinary}
We will say that an automorphic representation $\varPi$, with weight $(k'_1,k'_2)$, is \underline{$p$-ordinary} if $p \in S_{\mathrm{good}}$ and if the $p$-adic valuations (normalized so that the $p$-adic valuation of $p$ equals $1$) of the four roots of the Hecke polynomial $Q_{\varPi_p}(X)$ inside $\overline{\Q}_p$ (via $\iota_p$) equal $0$, $k'_2-2$, $k'_1-1$ and $k'_1 + k'_2 - 3$.  
\end{definition}

Assume now that $\varPi$ is $p$-ordinary. Let $\OO_\mathcal{E}$ denote the ring of integers in a finite extension $\mathcal{E}$ of $\Q_p$ with uniformizer $\pi_\mathcal{E}$. We then consider the Iwahori subgroup $\Iw_p(\OO_\mathcal{E})$ of $\GSp_4(\OO_\mathcal{E})$ which consists of all the matrices in  $\GSp_4(\OO_\mathcal{E})$ whose reduction to $\GSp_4(\OO_\mathcal{E}/(\pi_\mathcal{E}))$ is upper-triangular. Consider the commutative subalgebra $\Z[U_{p,1},U_{p,2}]$ of the $\Z$-algebra consisting of all $\Iw_p$ bi-invariant continuous functions $\GSp_4(\Q_p) \rightarrow \overline{\Q}$ with compact support $C_c^\infty(\GSp_4(\Q_p)//\Iw_p(\Z_p),\Z)$, where $U_{p,2},U_{p,1}$ are the characteristic functions of the following double cosets:

\begin{align*}
\Iw_p(\Z_p)\left[\begin{array}{cccc} p \\ & p \\ & & 1 \\ & & & 1 \end{array} \right] \Iw_p(\Z_p),  \qquad  \Iw_p(\Z_p)\left[\begin{array}{cccc} p^2 \\ & p \\ & & p \\ & & & 1 \end{array} \right] \Iw_p(\Z_p).
\end{align*}

The dimension $\varPi_p^{\Iw_p(\Z_p)}$ turns out to be $8$. Suppose $\alpha_0$ and $\alpha_1$ are two roots of the Hecke polynomial $Q_{\varPi_p}(X)$ such that $\val_{p}(\alpha_0) = 0$ and $\val_p(\alpha_1) = k'_2-2$. Following Miyauchi--Yamauchi \cite[Section 7.1]{MR3320491}, one can consider the $p$-ordinary $p$-stabilized vector. That is, there exists a $p$-ordinary $p$-stabilized vector\footnote{the $p$-ordinary $p$-stabilized vector is unique, up to scalars, when the weight is cohomological, that is, $k'_1\geq k'_2 \geq 3$.} in $\varPi_p^{\Iw_p(\Z_p)}$ that determines (and is characterized) by the following ring homomorphism: 
\begin{align}\label{eq:Upeigenvalues}
\phi_{p-\mathrm{stab,ord}}: 	\Z[U_{p,1},U_{p,2}] & \rightarrow \overline{\Q}, \\
\notag U_{p,2} &\mapsto \alpha_0, \\
\notag  U_{p,1} &\mapsto (p^{2-k'_2}\alpha_1)\alpha_0.
\end{align}
The $p$-ordinary $p$-stabilized vector $\vec{v}$ determines a ring homomorphism: 
\begin{align}\label{eq:ordstab}\phi_{\vec{v}}: \Z\left[\left\{U_{p,1},U_{p,2}\right\}  \cup \left\{T_{\ell,0},T_{\ell,0}^{-1}, T_{\ell,1},T_{\ell,2}\right\}_{\ell \in S_{\mathrm{good}}\setminus \{p\}} \right] \rightarrow  \overline{\Q}_p.
\end{align}

We will call $\phi_{\vec{v}}$ the \textit{Hecke eigensystem with weight $(k'_1,k'_2)$}. 

After combining the works of Laumon, Taylor, Weissauer and Urban, it follows that there exists a four dimensional Galois representation associated to $\varPi$ satisfying the following properties (see \cite[Theorem 3.1]{MR4454492} and its proof for more details).
\begin{theorem}[Laumon  \cite{MR2234859}, Taylor \cite{MR1240640},  Weissauer \cite{MR2234860}, \cite{MR2530981}, Urban \cite{MR2234861}]\label{thm:GSP4weissauertaylorlaumon} \mbox{}
Let $\varPi$ be an automorphic representation of $\GSp_4(\mathbb{A}_\Q)$ with regular \textit{cohomological} weight $(k'_1,k'_2)$, that is,  $k'_1\geq k'_2 \geq 3$.
\begin{enumerate}
\item\label{thm:galoispoint1} Let $E_\varPi$ denote the field obtained by adjoining to $\Q$, all roots of the Hecke polynomial $Q_{\varPi_\ell}(X)$ for every $\ell \in S_{\mathrm{good}}$. Then, $E_\varPi$ is a number field.
\item\label{thm:galoispoint2} There exists a unique continuous semi-simple Galois representation $\varrho_{\varPi,p} :G_{\Q,S} \rightarrow \GSp_4(E_{\varPi,p})$ such that for every prime $\ell \in S_{\mathrm{good}} \setminus \{p\}$, we have the following equality of characteristic polynomials:
\begin{align} \label{eq:equalityheckecharacteristicpoly}
\det\left(X \cdot I_4-\varrho_{\varPi,p}(\Frob_\ell)\right) = Q_{\varPi_\ell}(X).
\end{align}
Here, $E_{\varPi,p}$ denotes the compositum of the fields $\iota_p(E_\varPi)$ and $\Q_p$ inside $\overline{\Q}_p$. In particular, the Galois representation $\varrho_{\varPi,p}$ is unramified at all primes $\ell \in S_{\mathrm{good}} \setminus \{p\}$. 
\item\label{thm:galoispoint3} The similitude character of $\varrho_{\varPi,p}$ is $\chi_{\varPi}\cyc^{k'_1+k'_2-3}$, where $\chi_{\varPi}$ denotes the Galois character associated to the central character of $\Pi$ and $\cyc:G_{\Q,S} \rightarrow \Z_p^\times$ denotes the $p$-adic cyclotomic character.
\item\label{thm:galoispoint4} The Galois representation $\varrho_{\varPi,p}$ is Hodge--Tate  with weights (we normalize the $p$-adic cyclotomic character to have Hodge--Tate weight $1$)
\begin{align*}
0, \quad k'_2 -2, \quad k'_1-1, \quad k'_1+k'_2-3.
\end{align*}
\item\label{thm:galoispoint5} If $\varPi$ is $p$-ordinary and is neither CAP nor endoscopic, then there exist unramified characters $\psi_0,\psi_1, \psi_2,\psi_3$ of $G_{\Q_p}$ such that
\[\varrho_{\varPi,p}|_{G_{\Q_p}} \simeq \begin{pmatrix} \psi_3\cyc^{k'_1+k'_2-3} & \star & \star & \star \\ 0 & \psi_2\cyc^{k'_1-1} & \star & \star\\ 0 & 0 & \psi_1\cyc^{k'_2-2} & \star \\
0 & 0 & 0 & \psi_0 \end{pmatrix},\]
with $\psi_0(\Frob_p)=\alpha_0$, $\psi_1(\Frob_p) = p^{2-k'_2}\alpha_1$, $\psi_2(\Frob_p) = p^{1-k'_1}\alpha_2$ and $\psi_3(\Frob_p) = p^{3-k'_1-k'_2}\alpha_3$.
Here, $\alpha_0, \alpha_1, \alpha_2, \alpha_3$ are roots of the Hecke polynomial $Q_{\varPi_p}(X)$ with valuations $0$, $k'_2-2$, $k'_1-1$ and $k'_1 + k'_2 - 3$, respectively. 
\end{enumerate}
\end{theorem}

\subsection{Stable Yoshida lifts} \label{subsec:stableyoshidalift}

Let $g$ be a cuspidal Hilbert modular eigenform over $K$ with regular weight $(\kappa'_1,\kappa'_2)$ with $\kappa'_1 \geq \kappa'_2 \geq 2$, square-free level $\mathfrak{M}$ and trivial Nebentypus.  Let $\tau_{g}:G_{K,S} \rightarrow \Gl_2(\overline{\Z}_p)$ denote the two-dimensional $p$-adic Galois representation associated to $g$.  Note that $\det(\tau_{g})$ equals $\Res^{\Q}_{K}\left(\cyc^{\kappa'_1-1}\right)$.

\begin{definition}\label{def:ordgl2}
We say that $g$ is nearly ordinary at $\p_1$ and $\p_2$ if 
\begin{itemize}
    \item the $p$-adic valuations of the roots of the Hecke polynomial at $\p_1$ \[X^2-\iota_p(a(\p_1,g))X+p^{\kappa_1'-1}\] are $0$ and $\kappa'_1-1$, and
    \item the $p$-adic valuations of the roots of the Hecke polynomial at $\p_2$ \[X^2-\iota_p(a(\p_2,g))X+p^{\kappa'_1-1}\] 
    are $\frac{\kappa'_1-\kappa'_2}{2}$ and $\frac{\kappa'_1+\kappa'_2}{2}-1$.
\end{itemize} 
\end{definition}
Here, we assume that $p$ is a good prime, so that $(p,\mathfrak{M})=1$.  One could rephrase this definition as requiring that the Hodge polygons and the Newton polygons coincide. The definition could also be restated as requiring the eigenvalues of the normalized (as in \cite{MR960949}) Hecke operators $T_{\p_1}$ and $T_{\p_2}$ are $p$-adic units. 

We have the following proposition pertaining to the construction of an automorphic representation for ${\GSp_4}_{/\Q}$ corresponding to the automorphic induction of $g$.  The origins of this construction go back to Yoshida \cite{MR586427,MR758701} and B\"{o}cherer--Schulze-Pillot \cite{MR1096467,MR1292796,MR1475167}.

\begin{proposition}\label{thm:yoshidarepspace} 
Let $g$ be a Hilbert modular eigenform over $K$ as above.
There exists a cuspidal automorphic representation $\varPi_0$ with weight $(k'_1, k'_2)\coloneqq \left(\dfrac{\kappa'_1 + \kappa'_2}{2},\dfrac{\kappa'_1 - \kappa'_2}{2}+2\right)$ 
satisfying the following properties:
\begin{enumerate}

\item\label{con:stableYos1} For all primes $\ell \notin S \cup \{p\}$, the Hecke polynomial of $\varPi_0$ at $\ell$ equals $\det\left(X \cdot I_4-\Ind^{G_{\Q,S}}_{G_{K,S}}(\tau_{g})(\Frob_\ell)\right)$.

\item\label{con:stableYos3} There exists a fixed vector (say $f$) under the Siegel congruence subgroup $\Gamma_0^{(2)}(M\Delta_K)$ for $\varPi_0$. Here, $\Delta_K$ is the discriminant of the quadratic field $K$.
\item\label{con:stableYos4} If $g$ is nearly ordinary at $\p_1$ and $\p_2$, then  $\varPi_0$ is $p$-ordinary (in the sense of Definition \ref{def:gsppordinary}).

\end{enumerate}

\end{proposition}

\begin{proof}
For the construction of automorphic representation (\textit{associated to the stable Yoshida packet}) satisfying the properties prescribed in part (\ref{con:stableYos1}), see Roberts \cite[Theorem 8.6(1)]{MR1871665}. 

The existence of a Siegel modular form for $\varPi_0$ whose level equals the Siegel congruence subgroup $\Gamma_0^{(2)}(M\Delta_K)$, as in Part (\ref{con:stableYos3}), follows from work of Hsieh--Namikawa. See \cite[Section 3.7]{MR3623733} and \cite[Proposition 4.3]{MR3858470}. 

Since $p \in S_{\mathrm{good}}$ for $\varPi_0$, Part (\ref{con:stableYos4}) follows by comparing the Euler factor at $p$ for the $L$-function $L(\varPi_0,s)$ associated to $\varPi_0$ with the product of the Euler factors at $\p_1$ and $\p_2$ for the $L$-function $L_K(g,s)$ associated to $g$. Let $Q_{\varPi_0,p}(X)$ denote the degree $4$ Hecke polynomial associated to $\varPi_0$ at $p$. Let $ Q_{g,\p_i}$ denote the degree two Hecke polynomial associated to $g$ at $\p_i$, for $i \in \{1,2\}$. Since $p$ splits in $K$, one obtains an equality for all complex numbers $s$: 
\[p^{-4s}Q_{\varPi_0,p}(p^{s}) = \left(p^{-2s}Q_{g,\p_1}(p^{s})\right) \left( p^{-2s}Q_{g,\p_2}(p^{s})\right).\]
Therefore, the monic degree four polynomials $Q_{\varPi_0,p}(x)$ and $Q_{g,\p_1}(x)Q_{g,\p_2}(x)$ must coincide. Since $g$ is nearly ordinary at $\p_1$ and $\p_2$, the $p$-adic valuations of the roots of the product $Q_{g,\p_1}(x)Q_{g,\p_2}(x)$  (and hence of $Q_{\varPi_0,p}(x)$) equal
\begin{align} \label{eq:valuationofroots}
0, \qquad \underbrace{k'_2-2}_{\dfrac{\kappa'_1 - \kappa'_2}{2}}, \qquad \underbrace{k'_1-1}_{\dfrac{\kappa'_1 + \kappa'_2}{2}-1}, \qquad \underbrace{k'_1+k'_2-3}_{\kappa'_1-1}
\end{align}
This establishes the $p$-ordinarity of $\varPi_0$. 
\end{proof}

\begin{remark}
 If the stable Yoshida lift obtained in Proposition \ref{thm:yoshidarepspace} has cohomological weights (with $k'_1 \geq k'_2 \geq 3$) then the $p$-adic Galois representation $\varrho_{\varPi_0,p}$ attached to it in Theorem \ref{thm:GSP4weissauertaylorlaumon} is isomorphic to $\Ind_{G_{K,S}}^{G_{\Q,S}}(\tau_g)$. 
\end{remark}

\begin{remark}
For establishing parts (\ref{con:stableYos1}) and (\ref{con:stableYos3}) of Proposition \ref{thm:yoshidarepspace}, the assumption that $g$ is nearly ordinary is not required. 
\end{remark}

\subsection{Hida's $p$-ordinary Hecke algebra for $\GSp_4$} \label{subsec:hidagsp}

For each $i \in \{1,2\}$, we let $G_i$ denote $\Z_p^\times$. Let $\widetilde{\Lambda}$ denote the completed group ring $\Z_p\llbracket G_1 \times G_2 \rrbracket$. Consider  $G_1 \times G_2$ as a diagonal torus inside $\mathrm{Sp_4}(\Z_p)$: 
\begin{align*}
 G_1 \times G_2 &\hookrightarrow \mathrm{Sp_4}(\Z_p), \\
(g_1,g_2) & \mapsto \left[\begin{array}{cccc}g_1 & 0 & 0 & 0 \\ 0 & g_2 & 0 & 0 \\ 0 & 0 & g_2^{-1}& 0 \\ 0 & 0 & 0 & g_1^{-1} \end{array}\right]. 
\end{align*}

For each $(k'_1,k'_2) \in \Z^2$, the character
\begin{align*}
	G_1 \times G_2 & \rightarrow \Z_p^\times \\
	(g_1,g_2) &\mapsto  g_1^{k'_1} g_2^{k'_2}.
\end{align*}
determines a ring homomorphism $\varphi_{(k'_1,k'_2)}: \widetilde{\Lambda} \rightarrow \Z_p$ of completed $\Z_p$-algebras. We shall say that a prime ideal $\mathfrak{p}$ of $\widetilde{\Lambda}$ is a \textit{cohomological prime} with weight $(k'_1,k'_2)$ if it is the kernel of $\varphi_{(k'_1,k'_2)}$ with $k'_1\geq k'_2 \geq 3$. Note that the semi-local ring $\widetilde{\Lambda}$ can be decomposed into a product of local rings. Furthermore, these local rings are in one-to-one correspondence with group homomorphisms $(\upsilon^{a'},\upsilon^{b'}):\mu_{p-1} \times \mu_{p-1} \rightarrow \mathbb{F}_p^\times$. Here, the map $\upsilon:\mu_{p-1} \rightarrow \mathbb{F}_p^\times$ is obtained by composing the Hensel lift $\mu_{p-1} \hookrightarrow \mathbb{Z}_p^{\times}$ with the canonical surjection $\mathbb{Z}_p^{\times} \twoheadrightarrow \mathbb{F}_p^{\times}$. We call the corresponding local ring $\Lambda_{(a',b')}$. Observe that $\Lambda_{(a',b')}$ is isomorphic as a ring to the power series $\Z_p \llbracket y_1,y_2 \rrbracket$ in two variables. If $(k'_1,k'_2) \equiv (a',b') \pmod{(p-1)\Z^2}$, then the map $\varphi_{(k'_1,k'_2)}$ factors via $\Lambda_{(a',b')}$. The variables $y_1$ and $y_2$ are \textit{chosen} so that the morphism $\Lambda_{(a',b')} \rightarrow \overline{\Q}_p$, induced by $\varphi_{(k'_1,k'_2)}$, is defined by sending $y_i$ to $(1+p)^{k'_i}-1$ for $i \in \{1,2\}$. Suppose that $\mathcal{R}_2$ is a  reduced algebra that is finite and torsion-free over  $\Lambda_{(a',b')}$. We will say that a prime $\mathcal{P}$ of $\mathcal{R}_2$ is cohomological and has weight $(k'_1,k'_2)$ if the pullback of $\mathcal{P}$ to $\widetilde{\Lambda}$ equals $\ker(\varphi_{(k'_1,k'_2)})$. We let $\mathcal{X}_{\mathrm{coh}}(\mathcal{R}_2)$ denote the set of cohomological primes of $\mathcal{R}_2$. We also consider the following subset of cohomological primes:
\begin{multline}
\mathcal{X}^{\geq 4}_{\mathrm{temp}}(\mathcal{R}_2) \coloneqq \biggl\{\mathcal{P} \in \mathrm{Spec}(\mathcal{R}_2) : \mathcal{P} \text{ is a cohomological prime with } k'_1 > k'_2 \geq 4 \\[-1em]
\text{and } (k'_1,k'_2) \equiv (a',b') \pmod{(p-1)\mathbb{Z}^2}\biggr\}
\end{multline}

Recall from the introduction  that weight $(k_1,k_2)$ of the Yoshida lift $f_0$ fixes a congruence class $(a_0,b_0) \pmod{(p-1)\Z^2}$. We denote $\Lambda_{(a_0,b_0)}$ by $\Lambda$. By Proposition \ref{thm:yoshidarepspace}(\ref{con:stableYos4}), the automorphic representation corresponding to $f_0$ is $p$-ordinary. By Proposition \ref{thm:yoshidarepspace}(\ref{con:stableYos3}), this  automorphic representation is unramified at all the primes not dividing $M\Delta_K$. We recall the main vertical control theorem  independently due to Hida \cite[Theorem 1.1]{MR1954939} and Pilloni \cite[Theorem 7.1]{MR3059119}, \cite[Corollary 1.3]{MR2783930} and some of its consequences. See \cite[Section 2]{hsieh_yoshida}.

\begin{theorem}[Vertical control theorem (Hida, Pilloni)]\label{thm:verticalcontrol}\mbox{} 

There exists a finite, local, reduced, torsion-free and equidimensional (with dimension $3$) $\Lambda_{(a_0,b_0)}$-algebra $\TT$ and a ring homomorphism $j_4:\Z\left[\left\{U_{p,2},U_{p,1}\right\} \cup \left\{T_{\ell,2},  T_{\ell,1}, T_{\ell,0},T_{\ell,0}^{-1}\right\}_{(\ell,Mp\Delta_K)=1}\right] \rightarrow \TT$
satisfying the following properties:
\begin{enumerate}

\item\label{item:gsp4genbyheckeoperators} The image of $j_4$ is dense in $\TT$. 
\item\label{item:gsp4heckealgdensityprimes} $\mathcal{X}^{\geq 4}_{\mathrm{temp}}(\TT)$ is dense in $\mathrm{Spec}(\TT)$ with respect to the Zariski topology. 

\item \label{item:gsp4existence} There exists a ring homomorphism \[  \varPhi_{(k_1,k_2)}:\TT \rightarrow \overline{\Q}_p,\]
such that the induced map $\varPhi_{(k_1,k_2)} \circ j_4$ corresponds to the Hecke eigensystem of the ordinary $p$-stabilization of $f_0$. If $k_2 \geq 3$, then $\ker(\varPhi_{(k_1,k_2)})$ is cohomological with weight $(k_1,k_2)$.

\item\label{item:actualcontrol4} Suppose we are given a ring homomorphism $\varPhi_{(k'_1,k'_2)}:\TT \rightarrow \overline{\Q}_p$ such that $k'_1 \geq k'_2 \geq 4$. Then, the induced ring homomorphism $\varPhi_{(k'_1,k'_2)} \circ j_4$ corresponds to the Hecke eigensystem, as in equation (\ref{eq:ordstab}), of a non-zero $p$-ordinary stabilized vector of a $p$-ordinary automorphic representation $\varPi$  with level $\Gamma_0^{(2)}(M\Delta_K)$ and cohomological weight $(k'_1,k'_2)$.  If $k'_1 > k'_2 \geq 4$, then $\ker(\varPhi_{(k'_1,k'_2)}) \in \mathcal{X}^{\geq 4}_{\mathrm{temp}}(\TT)$. 
\item Suppose $\phi: \Z\left[\left\{U_{p,2},U_{p,1}\right\}  \cup \left\{T_{\ell,0},T_{\ell,0}^{-1}, T_{\ell,1},T_{\ell,2}\right\}_{(\ell,Mp\Delta_K)=1} \right] \rightarrow  \overline{\Q}_p$
corresponds to the Hecke eigensystem, as in equation (\ref{eq:ordstab}), of a non-zero $p$-ordinary stabilized vector of a $p$-ordinary automorphic representation $\varPi$  with level $\Gamma_0^{(2)}(M\Delta_K)$ and weight $(k'_1,k'_2)$  such that \begin{itemize}
\item $(a_0,b_0) \equiv (k'_1,k'_2) \mod (p-1)\Z^2$.
\item $k'_1 \geq k'_2 \geq 4$.
\end{itemize} Then, there exists a ring homomorphism \[  \varPhi_{(k'_1,k'_2)}:\TT \rightarrow \overline{\Q}_p,\]such that $\phi = \varPhi_{(k'_1,k'_2)} \circ j_4$. Furthermore, if $k'_1>k'_2 \geq 4$, then $\ker(\varPhi_{(k'_1,k'_2)}) \in \mathcal{X}^{\geq 4}_{\mathrm{temp}}(\TT)$.
\end{enumerate}

\end{theorem}

\section{Shapes and properties of Galois representations}\label{sec:galoisrep}
In this section, we establish the main properties of the Galois representation valued in the $\GSp_4$ ordinary Hecke algebra. These properties are used to get the surjective map $R^{\min} \to \TT$ mentioned in Theorem \ref{thm::rist}.

We first establish the symplectic property of the residual representation. 
Let $\{\vec{v}_1, \vec{v}_2\}$ denote a basis for the $\mathbb{F}[G_{K,S}]$ representation space corresponding to $\bar\tau$ (which was defined in \S~\ref{setting:all}). Choose a prime $\ell_0$ that is ramified in $K/\Q$; the hypothesis \ref{hyp:squarefree} tells us that the natural image of inertia subgroup $I_{\ell_0}$ of $G_{\Q,S}$ inside  $\Gal{{\overline{\Q}^{\mathrm{ker}(\bar\tau)} \cdot \overline{\Q}^{\mathrm{ker}(\bar\tau^c)}}}{\Q}$ has order two. Let $\sigma$ denote an element in $I_{\ell_0}$ that maps to the non-trivial order two element. Such an element exists due to hypothesis \ref {hyp:squarefree}.  Note that $\sigma$ does not belong to $G_{K,S}$, but $\sigma^2$ belongs to $G_{K,S}$.

We consider the action of $G_{\Q,S}$ on the ordered basis $\{\sigma\cdot \vec{v}_1,  \vec{v}_1,  \vec{v}_2,  \sigma\cdot \vec{v}_2\}$ of the induced representation $\Ind^{G_{\Q,S}}_{G_{K,S}}(\bar\tau)$ giving us a homomorphism $\bar\rho:G_{\Q,S} \rightarrow \Gl_4(\mathbb{F})$. For an element $h$ in $G_{K,S}$,  let $\bar\tau(h)= \begin{pmatrix} \bar{a}_h & \bar{b}_h \\ \bar{c}_h & \bar{d}_h \end{pmatrix}$ and $\underbrace{\bar\tau(\sigma^{-1}h\sigma)}_{\bar\tau^c(h)}= \begin{pmatrix} \bar{x}_h & \bar{y}_h \\ \bar{u}_h & \bar{v}_h \end{pmatrix}$. We then have
\begin{align} \label{eq:residualembeddingGL2GSP4}
    \bar\rho(\sigma)= \begin{pmatrix}
0 & 1 & 0 & 0 \\1 & 0 & 0 & 0\\ 0 & 0 & 0 & 1 \\ 0 & 0 & 1 & 0
\end{pmatrix}, \qquad \bar\rho(h) = \begin{pmatrix} 
\bar{x}_h & 0 & 0 & \bar{y}_h \\ 0 & \bar{a}_h & \bar{b}_h& 0 \\ 0 & \bar{c}_h & \bar{d}_h & 0 \\ \bar{u}_h & 0 & 0 & \bar{v}_h  \end{pmatrix}.
\end{align}
 
Since the character $\det(\bar\tau)$ is obtained as the restriction of a $G_{\Q,S}$-character, we see that $\mathrm{Image}(\bar\rho)$ lands inside $\GSp_4(\mathbb{F})$. Let $\bar\lambda:G_{\Q,S} \rightarrow \mathbb{F}^\times$ denote the similitude character obtained from $\bar\rho$.
Let $\omega$ denote the Hensel lift of the mod-$p$ cyclotomic character.

\begin{theorem} \label{thm:symplecticbasis} \mbox{}
\begin{enumerate}
\item\label{hyp:symppairing} There exists a representation
\begin{align*}
\rho_{\TT} : G_{\Q,S} \to \GSp_4(\TT)
	\end{align*}
    lifting $\bar\rho$ such that for all $\ell \notin S$,  we have $\mathrm{Trace}({\rho}_{\TT}(\Frob_\ell))=T_{\ell,2}$.
  Furthermore, for every element $g$ in the decomposition group $\Gal{\overline{\Q}_p}{\Q_p}$,
{\allowdisplaybreaks\begin{align*}
  \rho_{\TT}(g) &= \left(\begin{array}{cccc}  \dfrac{\widetilde{\lambda}}{\widetilde{\psi}_0}(g)  & \star & \star  & \star \\ 0 & \dfrac{\widetilde{\lambda}}{ \widetilde{\psi}_1\cyc^{-2}\omega^{b_0}\widetilde{\kappa}_2}(g) & \star & \star \\ 0 & 0   & \widetilde{\psi}_1\cyc^{-2}\omega^{b_0}\widetilde{\kappa}_2(g) & \star \\ 0 & 0 & 0 & \widetilde{\psi}_0(g) \end{array} \right).
\end{align*}
}
Here, $\widetilde{\lambda}:G_{\Q,S} \rightarrow \TT^\times$  is a global character uniquely characterized by sending $\Frob_\ell$ to $\ell^3 T_{\ell,0}$ for all $\ell$ not dividing $M\Delta_K$. Here, \[\widetilde{\kappa}_2:G_{\Q_p} \rightarrow \Gal{\Q_p^{\mathrm{cyc}}}{\Q_p} \cong 1+p\Z_p \hookrightarrow \Z_p\llbracket1+p\Z_p\rrbracket^\times \cong \Z_p\llbracket y_2\rrbracket^\times\] is the tautological character coming from the second weight variable. Furthermore, for each $i \in \{0,1\}$, the characters $\tilde{\psi}_i:G_{\Q_p} \rightarrow \TT^\times$ are unramified with \[\tilde{\psi}_0(\Frob_p)=U_{p,2}, \qquad \tilde{\psi}_1(\Frob_p)=U_{p,1}/U_{p,2}.\]
\item\label{hyp:localshaperamlevel}  If $\ell$ is a prime dividing $M$, then there exists a matrix $P_{\ell} \in \GSp_4(\TT)$ such that $ {\rho}_{\TT}(I_{\ell})$ is the pro-cyclic group generated by \[P_{\ell}\begin{pmatrix} 1 & 1 & 0 & 0 \\ 0 & 1 & 0 & 0\\0 & 0 & 1 & -1 \\ 0 & 0 & 0 & 1\end{pmatrix}P_{\ell}^{-1}.\]

\item\label{hyp:localshapediscK} If $\ell$ ramifies in $K$, then there exists a matrix $P_\ell \in \GSp_4(\TT)$ such that $ {\rho}_{\TT}(I_{\ell})$ is the cyclic group generated by \[P_{\ell}\begin{pmatrix} 1 & 0 & 0 & 0 \\ 0 & -1 & 0 & 0\\0 & 0 & -1 & 0 \\ 0 & 0 & 0 & 1\end{pmatrix}P_{\ell}^{-1}.\]

\end{enumerate}
\end{theorem}

\begin{proof}
To prove point (\ref{hyp:symppairing}), we follow the arguments of Genestier--Tilouine \cite[Lemma 4.3.3]{MR2234862} and Pilloni \cite{MR2920881}.

Since the traces of $p$-adic Galois representations attached to Siegel cuspidal eigenforms (as in Theorem \ref{thm:GSP4weissauertaylorlaumon}) are pseudocharacters of dimension four (e.g., see \cite[Section 1.2]{MR2656025} for the definition of pseudocharacters), one can use a density argument (e.g., see \cite[Theorem A2]{hsieh_yoshida} and its proof for such an application of a density argument) using the cohomological specializations in the set $\mathcal{X}^{\geq 4}_{\mathrm{temp}}(\TT)$ to show that there exists a unique continuous pseudocharacter $\pcT : G_{\Q,S} \rightarrow \TT$ of dimension four, characterized by the following assignment
\begin{align*}
\Frob_\ell \mapsto T_{\ell,2}, \quad \forall \ell \notin S.
\end{align*}
Therefore, the residual pseudocharacter $G_{\Q,S} \xrightarrow{\pcT} \TT \rightarrow \mathbb{F}$ is the trace of an irreducible representation $\bar\rho:G_{\Q,S} \rightarrow \Gl_4(\mathbb{F})$. Hence, it follows from works of Nyssen and Rouquier \cite{MR1411348,MR1378546} that there exists a Galois representation $\varrho:G_{\Q,S} \rightarrow \Gl_4(\TT)$ lifting $\bar\rho$  such that $\tr(\varrho) = \pcT$. 

The definition of the Hecke operator $T_{\ell,0}$ and a density argument using the cohomological specializations in the set $\mathcal{X}^{\geq 4}_{\mathrm{temp}}(\TT)$ together allow us first to show that there exists a character $\widetilde{\lambda}: G_{\Q,S} \rightarrow \TT$ interpolating the similitude characters of the cohomological specializations. It is characterized by the following assignment:
\begin{align*}
\Frob_\ell \mapsto \ell^3 T_{\ell,0}, \quad \forall \ell \notin S.
\end{align*}
For instance, see \cite[Lemma 4.2]{hsieh_yoshida} and its proof. We note that the similitude character $\tilde{\lambda}$ lifts the residual character $\bar\lambda:G_{\Q,S} \rightarrow \mathbb{F}^\times$ given above since the $p$-adic Galois representation associated to any cohomological specialization in $\mathcal{X}^{\geq 4}_{\mathrm{temp}}(\TT)$ lifts $\bar\rho$. We also observe that its square equals  $\det(\varrho)$.

Let $V$ denote the underlying rank four $\TT$-module for the Galois representation $\varrho$.  Moreover, a similar density argument tells us that $\varrho$ is \textit{auto-dual}, i.e., we have the following isomorphism of $\TT[G_{\Q,S}]$-modules:
\begin{align} \label{eq:autoduality}
V (\widetilde{\lambda}^{-1}) &\cong \Hom_{\TT}(V,\TT),
\\ \notag
v &\mapsto \psi_{v}
\end{align}
Equation (\ref{eq:autoduality}) allows us to define three pairings $V \times V \rightarrow \TT$ as follows:
\begin{align} \label{eq:defnpairings}
\langle v_1, v_2 \rangle = \psi_{v_1}(v_2), \quad \langle v_1, v_2 \rangle_{\mathrm{alt}} = \psi_{v_1}(v_2) - \psi_{v_2}(v_1), \quad \langle v_1, v_2 \rangle_{\mathrm{sym}} = 2\langle v_1, v_2 \rangle - \langle v_1, v_2 \rangle_{\mathrm{alt}}.
\end{align}
One observes that for all $g \in G_{\Q,S}$, we have 
\begin{align} \label{eq:actionofGpairing}
  \langle gv_1, gv_2\rangle =\widetilde{\lambda}(g) \langle v_1, v_2\rangle, \quad \langle gv_1, gv_2\rangle_{\mathrm{alt}} = \widetilde{\lambda}(g) \langle v_1, v_2\rangle_{\mathrm{alt}}, \quad \langle gv_1, gv_2\rangle_{\mathrm{sym}} = \widetilde{\lambda}(g) \langle v_1, v_2\rangle_{\mathrm{sym}}.  
\end{align}
  Since $\varrho$ lifts $\bar\rho$, we consider, without loss of generality, a $\TT$-basis $\{E_1, E_2, E_3, E_4\}$ of $V$ lifting the standard symplectic basis $\left\{e_1, e_2, e_3, e_4\right\}$ of $\bar\rho$. We make the following claim: 

\begin{claim} \label{claim:tressymp}
Suppose $\langle \cdot,\cdot \rangle_{0}$ is a non-degenerate bilinear pairing on  $\mathbb{F}^4$ such that  \[\langle \bar\rho(g) v_1,\bar\rho(g) v_2 \rangle_{0}=\bar{\lambda}(g)\langle  v_1,v_2 \rangle_{0}, \qquad \forall \  g \in G_{\Q,S}, \quad  \forall \ v_1,v_2 \in \mathbb{F}^4.\]  Then, the pairing $\langle \cdot,\cdot \rangle_{0}$ is symplectic and equal to a scalar multiple of the standard symplectic pairing provided by the  matrix $J$ in $\mathbb{F}$.
\end{claim}
The isomorphism in equation (\ref{eq:autoduality}) tells us that the pairing $\langle \cdot,\cdot\rangle$, defined in equation (\ref{eq:defnpairings}), is non-zero modulo the maximal ideal of $\TT$. We note that the character $\tilde{\lambda}$ afforded by the pairings $\langle \cdot, \cdot \rangle$, $\langle \cdot, \cdot\rangle_{\mathrm{alt}}$ and $\langle \cdot, \cdot\rangle_{\mathrm{sym}}$, as in equation (\ref{eq:actionofGpairing}), modulo the maximal ideal of $\TT$ equals $\bar\lambda$. Claim \ref{claim:tressymp} allows us to deduce that the reduction of the symmetric pairing $\langle \cdot, \cdot \rangle_{\mathrm{sym}}$ modulo the maximal ideal of $\TT$ is zero. The definition of the alternating pairing $\langle \cdot, \cdot \rangle_{\mathrm{alt}}$ as given in equation (\ref{eq:defnpairings}) along with claim \ref{claim:tressymp}  tells us that the reduction of the alternating pairing $\langle \cdot, \cdot \rangle_{\mathrm{alt}}$ modulo the maximal ideal of $\TT$ is a non-zero scalar multiple of the symplectic pairing provided by the pairing matrix $J$ in $\mathbb{F}$.

\begin{proof}[Proof of Claim \ref{claim:tressymp}]
 Hypothesis \ref{hyp:bigimage} along with the embedding of the image of the $G_{K,S}$ representation $\bar\tau \oplus \bar\tau^c$ into $\GSp_4(\mathbb{F})$ given in equation (\ref{eq:residualembeddingGL2GSP4}) tells us that the matrices
\begin{align*}
 \begin{pmatrix} 1 & 0 & 0 & 1\\ 0 & 1 & 0 & 0 \\ 0 & 0 & 1 & 0 \\ 0 & 0 & 0 & 1 \end{pmatrix}, \quad \begin{pmatrix} 1 & 0 & 0 & 0\\ 0 & 1 & 0 & 0 \\ 0 & 0 & 1 & 0 \\ 1 & 0 & 0 & 1 \end{pmatrix}, \quad  \begin{pmatrix} 1 & 0 & 0 & 0\\ 0 & 1 & 1 & 0 \\ 0 & 0 & 1 & 0 \\ 0 & 0 & 0 & 1 \end{pmatrix}, \begin{pmatrix} 1 & 0 & 0 & 0\\ 0 & 1 & 0 & 0 \\ 0 & 1 & 1 & 0 \\ 0 & 0 & 0 & 1 \end{pmatrix},
\end{align*}
belong to the natural image of $\mathrm{SL}_2(\mathbb{F}_p) \times \mathrm{SL}_2(\mathbb{F}_p)$ inside $\mathrm{Image}(\bar\rho)$. Note that the restriction of the character $\bar{\lambda}$ to $\mathrm{SL}_2(\mathbb{F}_p) \times \mathrm{SL}_2(\mathbb{F}_p)$ is trivial by the construction given just before the theorem. Using these matrices, we obtain that the pairing $\langle \cdot,\cdot \rangle_{0}$  satisfies the following relations
\begin{align*}
  \langle e_1,e_1 \rangle_{0}  = \langle e_2,e_2 \rangle_{0}  = \langle e_3,e_3 \rangle_{0}  = \langle e_4,e_4 \rangle_{0}  =   \langle e_1,e_2 \rangle_{0} =  \langle e_1,e_3 \rangle_{0} =  \langle e_2,e_4 \rangle_{0} =  \langle e_3,e_4 \rangle_{0} = 0.
\end{align*} 
Similarly, using the following matrices in $\mathrm{SL}_2(\mathbb{F}_p) \times \mathrm{SL}_2(\mathbb{F}_p)$ inside $\mathrm{Image}(\bar\rho)$:
\begin{align*}
\begin{pmatrix} 1 &0 & 0 &1 \\ 0 &1 & 0 & 0 \\ 0 & 0 &1 &0 \\ 1 & 0 & 0 &2 \end{pmatrix}, \quad \begin{pmatrix} 1 &0 & 0 &0 \\ 0 &1 & 1 & 0 \\ 0 & 1 &2 &0 \\ 0 & 0 & 0 &1\end{pmatrix}
\end{align*}
we may conclude that $\langle e_1,e_4 \rangle_{0}=-\langle e_4,e_1 \rangle_{0}$ and $\langle e_2,e_3 \rangle_{0}=-\langle e_3,e_2 \rangle_{0}$. Combining our observations tells us that the pairing $\langle\cdot, \cdot \rangle_{0}$ is alternating. 
We also note that the matrix $\begin{pmatrix}
0 & 1 & 0 & 0 \\1 & 0 & 0 & 0\\ 0 & 0 & 0 & 1 \\ 0 & 0 & 1 & 0
\end{pmatrix}$ belongs to $\mathrm{Image}(\bar\rho)$ (see equation (\ref{eq:residualembeddingGL2GSP4})) and lies in $\ker(\bar\lambda)$. Using this matrix allows us to conclude that $\langle e_1,e_4 \rangle_{0}= \langle e_2,e_3 \rangle_{0}$. This allows us to conclude that the pairing matrix for  $\langle \cdot,\cdot \rangle_{0}$ equals $\langle e_1,e_4 \rangle_{0}J$. This completes the proof of the claim. 
\end{proof}

To complete the proof of point \ref{hyp:symppairing} of Theorem \ref{thm:symplecticbasis}, we proceed as follows. As $\bar\rho$ is absolutely irreducible, it follows that the cohomological specializations in $\mathcal{X}^{\geq 4}_{\mathrm{temp}}(\TT)$ are neither CAP nor endoscopic. Combining the $p$-distinguished hypothesis, Parts~\ref{thm:galoispoint3} and \ref{thm:galoispoint5} of Theorem~\ref{thm:GSP4weissauertaylorlaumon} and the density of $\mathcal{X}^{\geq 4}_{\mathrm{temp}}(\TT)$ in $\mathrm{Spec}(\TT)$,  we conclude that there exists a change-of-basis matrix $P_1$ in $\Gl_4(\TT)$ so that the restriction $P_1 \varrho P_1^{-1} \mid_{G_{\Q_p}}$ is upper-triangular with characters on diagonal as given in point \ref{hyp:symppairing}. For example, see the arguments provided in the proof of \cite[Theorem 4.6(7)]{hsieh_yoshida}. We can also find a matrix $\overline{P}_2$ in $\Gl_4(\mathbb{F})$ so that $\overline{P_2P_1 \varrho P_1^{-1}P_{2}^{-1}}$ equals $\bar\rho$. Note that the $p$-distinguished hypothesis forces $\overline{P_2}$ to be upper-triangular. Now consider an upper triangular matrix $P_2$ in $\Gl_4(\TT)$ that lifts $\overline{P_2}$. Note that the restriction $P_2P_1 \varrho P_1^{-1}P_{2}^{-1}$ lifts $\bar\rho$ and its restriction to $G_{\Q_p}$ is upper-triangular. Therefore, after scaling appropriately, the pairing matrix $J'$ for the symplectic pairing $\langle \cdot, \cdot \rangle_{\mathrm{alt}}$ with respect to this new basis --- given by the change of basis matrix $P_1^{-1}P_{2}^{-1}$ --- reduces to $J$ modulo the maximal ideal. In fact, the $p$-distinguished hypothesis along with the explicit description of the characters along the main diagonal in the local shape provided in the point \ref{hyp:symppairing} of the theorem allows us to write down 
\[J'=\begin{pmatrix} 0 & 0 & 0 & u_1 \\ 0 & 0 & u_2 & x_1 \\ 0 & -u_2 & 0 & x_2 \\ -u_1 & -x_1 & -x_2 & 0  \end{pmatrix}, \text{ with } u_1-1,u_2-1,x_1,x_2 \text{ in the maximal ideal of $\TT$}.\] One can then apply a further change of basis matrix $P_3 = \begin{pmatrix} 1 & -x_1/u_1 & -x_2/u_1u_2 & 0 \\ 0 & 1 & 0 & 0 \\ 0 & 0 & 1/u_2 & 0 \\ 0 & 0 & 0  & 1/u_1\end{pmatrix}$ which is upper-triangular, satisfies $P_3^TJ'P_3=J$ and the image of $P_3$ in $\Gl_4(\mathbb{F})$ is the identity matrix. One now directly verifies that the Galois representation $P_0 \varrho P_0^{-1}$ satisfies the hypothesis of point (\ref{hyp:symppairing}), with $P_0=P_3^{-1}P_2P_1$.

Point (\ref{hyp:localshaperamlevel}) follows from Proposition \ref{generallemma:elldividinglevel}. One applies Proposition \ref{generallemma:elldividinglevel}\ref{item:partareducedoverO} to all cohomological primes of $\mathcal{X}^{\geq 4}_{\mathrm{temp}}(\TT)$ and a density argument to show that the conclusion in equation (\ref{eq:jordanblockformsquare}) holds for $\TT$. One then deduces point (\ref{hyp:localshaperamlevel}) from Proposition \ref{generallemma:elldividinglevel}\ref{item:jordancanonicalformcharequation}. 

Point (\ref{hyp:localshapediscK}) can similarly be deduced from  proposition \ref{generallemma:elldividingdiscK}.  This concludes the proof of Theorem \ref{thm:symplecticbasis}.
\end{proof}

\subsection{Shape of the local Galois representations at bad primes}
Let $\mathrm{O}$ denote the Witt ring $W(\mathbb{F})$. Let $g_\ell$ denote a lift in $I_\ell$ of the topological generator of the pro-$p$ part (which is isomorphic to $\Z_p$)  of the tame inertia group at $\ell$. 

\begin{lemma} \label{lemma:residualshapedividinglevel}Suppose $\ell$ is a prime of $\Q$ dividing $M$. Let $\ell_0$ denote a prime of $K$ lying above $\ell$. 
\begin{enumerate}[label=(\roman*)]
\item\label{item:partibartau} Then, $\bar\tau \mid_{G_{K_{\ell_0}}}$ and $\bar\tau^c \mid_{G_{K_{\ell_0}}}$ are ramified. Furthermore, there exist unramified characters \[\bar\chi_1,\bar\chi_2,\bar\chi^c_1,\bar\chi^c_2:G_{K_{\ell_0}} \rightarrow \overline{\mathbb{F}}_p^\times,\] such that we have an isomorphism of $\overline{\mathbb{F}}_p$-representations:
\begin{align} \label{eq:basischangebartau}
\bar\tau\mid_{G_{K_{\ell_0}}} \simeq \begin{pmatrix} \bar\chi_1 & * \\ 0 & \bar\chi_2 \end{pmatrix}, \qquad \bar\tau^c\mid_{G_{K_{\ell_0}}}\simeq \begin{pmatrix} \bar\chi^c_1 & * \\ 0 & \bar\chi^c_2 \end{pmatrix},
\end{align}
with the ratios $\dfrac{\bar\chi_1}{\bar\chi_2}$ and $\dfrac{\bar\chi^c_1}{\bar\chi^c_2}$ equal to the cyclotomic character. 
\item\label{item:partiibartau} We have an isomorphism of $\overline{\mathbb{F}}_p$-representations:
\begin{align} \label{eq:barrhoshapelocaldividingM}
\bar\rho \mid_{G_{K_{\ell_0}}} \simeq \begin{pmatrix} \bar\chi_1 & \star & 0 & 0 \\ 0 & \bar\chi_2 & 0 & 0 \\ 0 & 0 & \bar\chi_1^c & \star \\ 0 & 0 & 0 & \bar\chi_2^c\end{pmatrix}.  
\end{align}
\begin{itemize}
\item  If $\ell$ is a prime of $\Q$ that splits over $K$, then we have the following relations:
\begin{align} \label{eq:relationssplitcase}
\bar\chi_1^2 = (\bar\chi^c_1)^2, \quad \bar\chi_2^2 = (\bar\chi^c_2)^2, \quad \dfrac{\bar\chi_1}{\bar\chi_2}(\Frob_{\ell_0}) = \dfrac{\bar\chi^c_1}{\bar\chi^c_2}(\Frob_{\ell_0}) = \overline{\ell}.
\end{align}
\item If $\ell$ is a prime of $\Q$ that remains inert in $K$, then we have the following relations:
\begin{align} \label{eq:relationsinertcase}
\bar\chi_1 = \bar\chi^c_1, \quad \bar\chi_2 = \bar\chi^c_2, \quad \dfrac{\bar\chi_1}{\bar\chi_2}(\Frob_{\ell_0}) = \dfrac{\bar\chi^c_1}{\bar\chi^c_2}(\Frob_{\ell_0}) = \overline{\ell^2}.
\end{align}
\end{itemize}
In particular, the image of $\bar\rho \mid_{I_{\ell_0}}$ is a cyclic group of order $p$ and that \[\bar\rho(g_{\ell}) \sim_{\overline{\mathbb{F}}_p}  \begin{pmatrix} 1 & 1 & 0 & 0 \\ 0 & 1 & 0 & 0 \\ 0 & 0 & 1 & 1 \\ 0 & 0 & 0 & 1\end{pmatrix}.\] 
\end{enumerate}
\end{lemma}

\begin{proof}
Recall that $\tau_0:G_{K,S} \rightarrow \Gl_2(\overline{\Z}_p)$ is the two-dimensional $p$-adic Galois representation associated to the Hilbert cuspidal eigenform $g_0$. Using the local-global compatibility results of Carayol  \cite{MR870690} for twisted Steinberg representations, one obtains that the restriction $\tau_0\mid_{G_{K_{\ell_0}}}$ is ramified and has the following shape for all  $h \in I_{\ell_0}$:
\begin{align*}
\tau_0(h) = C \begin{pmatrix} 1& * \\ 0 & 1 \end{pmatrix} C^{-1}, \text{ for some $C$ in $\Gl_2(\overline{\Q}_p)$}.
\end{align*}

These observations allow us to conclude that the image of the restriction $\tau_0\mid_{I_{{\ell_0}}}$ is pro-$p$ and hence isomorphic to $\Z_p$. Hypothesis \ref{hyp:squarefree} tells us that $I_{\ell_0} \cong I_{\ell}$. These observations allow us to conclude the following equality of matrices over $M_2(\overline{\Z}_p)$:
\begin{align}
(\tau_0(g_\ell)-I_2)^2 = 0.
\end{align}
Considering their reductions in $\overline{\mathbb{F}}_p$ and using hypothesis \ref{hyp:optimallevel}, we get that
\begin{align}
\bar\tau(g_\ell)-I_2 \neq 0, \qquad (\bar\tau(g_\ell)-I_2)^2 = 0.
\end{align}
The relation between tame inertia and Frobenius at $\ell_0$ now provides us the desired unramified characters  $\bar\chi_1,\bar\chi_2:G_{K_{\ell_0}} \rightarrow \overline{\mathbb{F}_p}^\times$ for  $\bar\tau \mid_{G_{K_{\ell_0}}}$. We also note that the product $\bar\chi_1\bar\chi_2$ equals a power of the mod-$p$ cyclotomic character $\bar\epsilon^{\kappa_1-1}$. We note that  $\bar\tau^c \mid_{G_{K_{\ell_0}}}$ is isomorphic to $\bar\tau \mid_{G_{K_{\ell^c_0}}}$, where $\ell^c_0$ is the prime in $K$ lying over $\ell$ that is conjugate to $\ell_0$. If the prime $\ell$ of $\Q$ lying below $\ell_0$ is split in $K$, then the above conclusion follows by definition. Otherwise, we note that if $\ell$ is a prime of $\Q$ that remains inert in $K$, the claimed isomorphism utilizes the unipotent shape of $\bar\tau \mid_{I_{\ell_0}}$ along with the relation between tame inertia and Frobenius at $\ell$. The conclusion of part \ref{item:partibartau} follows for $\bar\tau^c \mid_{G_{K_{\ell_0}}}$ since $\ell^c_0$ also satisfies the same hypotheses as $\ell_0$.  This proves part \ref{item:partibartau}. 

The construction of $\bar\rho$ given above provides us the following isomorphism:
\begin{align} \label{eq:decompositionlocaldividingM}
\bar\rho \mid_{G_{K_{\ell_0}}} \cong \bar\tau \mid_{G_{K_{\ell_0}}} \oplus \bar\tau^c \mid_{G_{K_{\ell_0}}}. 
\end{align} 
Combining all our observations above and the isomorphism in equation (\ref{eq:decompositionlocaldividingM}) allows us to deduce the conclusion in part \ref{item:partiibartau}. Note that if $\ell$ is a prime of $\Q$ that remains inert in $K$, we have used the isomorphism $\bar\tau \mid_{G_{K_{\ell_0}}} \cong \bar\tau^c \mid_{G_{K_{\ell_0}}}$.
\end{proof}

\begin{proposition} \label{generallemma:elldividinglevel}
Suppose $\ell$ is a prime of $\Q$ dividing $M$.  Let $R$ denote a complete local Noetherian $\mathrm{O}$-algebra with finite residue field $\mathbb{F}$.  Let $\rho:\Gal{\overline{\Q}}{\Q} \rightarrow \GSp_4(R)$ denote a lift of $\bar\rho$. 
\begin{enumerate}[label=(\roman*)]
    \item\label{item:partareducedoverO} If $R$ is a reduced ring that is finite and torsion-free over $\Z_p$, then 
    \begin{align} \label{eq:jordanblockformsquare}
    (\rho(g_\ell)-I_4) &\neq 0, \\  (\rho(g_\ell)-I_4)^2 &= 0. \notag
    \end{align}
    \item\label{item:jordancanonicalformcharequation} Suppose that the conclusion in equation (\ref{eq:jordanblockformsquare}) holds. Then, there exists a matrix $P_{\rho,\ell}$ in $\GSp_4(R)$ such that $\rho(I_{\ell})$ is a pro-cyclic group generated by $P_{\rho,\ell}\begin{pmatrix} 1 & 1 & 0 & 0 \\ 0 & 1 & 0 & 0\\0 & 0 & 1 & -1 \\ 0 & 0 & 0 & 1\end{pmatrix}P_{\rho,\ell}^{-1}$.
\end{enumerate}
\end{proposition}
\begin{proof}
Since $R$ is reduced, to prove part \ref{item:partareducedoverO}  without loss of generality we may assume that $R$ equals a subring $\mathrm{O}'$ of a finite extension of $\Q_p$. First note that $\rho(I_\ell)$ is a pro-$p$-group since $\bar\rho(I_\ell)$ is a $p$-group by Lemma \ref{lemma:residualshapedividinglevel} along with the fact that $\ker\left(\Gl_4(\mathrm{O}') \rightarrow \Gl_4(\mathbb{F}) \right)$ is pro-$p$. Therefore, $\rho \mid_{I_\ell}$ is tamely ramified and pro-cyclic generated by $\rho(g_\ell)$.
\begin{claim}\label{claim:1isonlyeigenvalue}
The only eigenvalue of $\rho(g_\ell)$ is $1$.
\end{claim}

\begin{proof}[Proof of Claim \ref{claim:1isonlyeigenvalue}]
Suppose $\alpha$ is an eigenvalue of $\rho(g_\ell)$ that is not equal to $1$. The relationship between Frobenius at $\ell$ and tame inertia tells us that $\alpha^\ell$ is also an eigenvalue of $\rho(g_\ell)$. This in turn tells us that $\alpha$ is a root of unity. From the residual shape of $\bar\rho$ provided in Lemma \ref{lemma:residualshapedividinglevel}, we obtain that $\alpha$ must be a $p$-power root of unity. Hypothesis \ref{hyp:congruence1} and the fact that we are dealing with $4 \times 4$ matrices then tells us that the only possibility is $\alpha^{\ell^3-1}=1$. Let $V_{\alpha}, V_{\alpha^\ell}, V_{\alpha^{\ell^2}}$ denote the eigenspaces for the matrix $\rho(g_{\ell})$ over $\overline{\Q}_p$ for the eigenvalues $\alpha$, $\alpha^\ell$ and $\alpha^{\ell^2}$ respectively. We note that $\Frob_{\ell}$ permutes each of the eigenspaces $V_{\alpha}, V_{\alpha^\ell}, V_{\alpha^{\ell^2}}$ and leaves their direct sum $V_{\alpha} \oplus V_{\alpha^\ell} \oplus V_{\alpha^{\ell^2}}$ invariant inside the $\overline{\Q}_p$-representation afforded by $\rho$. Therefore they must all be of dimension $1$. We also note that the action of $\Frob^3_\ell$ on the direct sum is via multiplication-by-a-scalar, say $c_0$. This observation allows us to conclude that the action of (any lift of) $\Frob^2_\ell$ on the $\overline{\Q}_p$-representation space afforded by $\rho$ has three residually distinct eigenvalues (corresponding to the cube-roots of $c_0^2$). However, Lemma \ref{lemma:residualshapedividinglevel} tells us that (any lift of) $\Frob_\ell^2$ has only two distinct eigenvalues for its action on the $\overline{\mathbb{F}}_p$ representation afforded by $\bar\rho$. This contradicts the existence of an eigenvalue $\alpha$ of $\rho(g_\ell)$ that is not equal to $1$. This concludes the proof of Claim \ref{claim:1isonlyeigenvalue}.
\end{proof}
Let $X$ denote $\rho(\Frob_\ell)$, for some lift $\Frob_\ell$ of Frobenius at $\ell$. Since $\bar\rho(g_\ell)$ is not identity by Lemma \ref{lemma:residualshapedividinglevel}, to complete the proof of part \ref{item:partareducedoverO}, we need to rule out the following Jordan canonical forms for $\rho(g_\ell)$: 
\begin{align*}
\begin{pmatrix} 1 & 1 & 0 & 0 \\ 0 & 1 & 1 & 0 \\ 0 & 0 & 1 & 1 \\ 0 & 0 & 0 & 1    \end{pmatrix}, \qquad \text{or} \qquad   \begin{pmatrix} 1 & 1 & 0 & 0 \\ 0 & 1 & 1 & 0 \\ 0 & 0 & 1 & 0 \\ 0 & 0 & 0 & 1    \end{pmatrix}.
\end{align*}
Let $B$ denote the Jordan canonical form of $\rho(g_\ell)$. One can solve for a general formula for $X$ so that the relation $XB = B^\ell X$ holds. Let $a,b,c,d$ denote the eigenvalues of $X$ in $\overline{\Q}_p$. In the above cases, the general formula yields (after suitable reordering) that
\begin{align*}
    a/b=b/c=c/d=\ell, \qquad \text{or} \qquad    a/b=b/c=\ell. 
\end{align*}
Therefore, residually one obtains
\begin{align*}
    \overline{a/b}^2=\overline{b/c}^2=\overline{c/d}^2=\overline{\ell}^2, \qquad \text{or} \qquad    \overline{a/b}^2=\overline{b/c}^2=\overline{\ell}^2. 
\end{align*}
In both of these cases, one would obtain that $X^2$ has at least three distinct residual eigenvalues. However, this is prohibited by Lemma \ref{lemma:residualshapedividinglevel}. Therefore, $(\rho(g_\ell)-I_4) \neq 0$ and $(\rho(g_\ell)-I_4)^2 = 0$. This completes the proof of part \ref{item:partareducedoverO}. Part \ref{item:jordancanonicalformcharequation}  follows from \cite[Lemma 9.2.1]{MR2234862}.
\end{proof}

\begin{lemma}\label{lemma:ramifieddiscshape} Suppose $\ell$ is a prime of $\Q$ dividing $\Delta_K$. Let $\ell_0$ denote the unique prime of $K$ lying above $\ell$. 
\begin{enumerate}
\item\label{item:part1ramifieddisc}  $\bar\tau \mid_{G_{K_{\ell_0}}}$ is unramified. Furthermore, there exist unramified characters $\bar\chi_1,\bar\chi_2:G_{K_{\ell_0}} \rightarrow \overline{\mathbb{F}}_p^\times$, such that we have an isomorphism of $\overline{\mathbb{F}}_p$-representations:
\begin{align} \label{eq:basischangebartaudeltak}
\bar\tau \simeq \begin{pmatrix} \bar\chi_1 & \star \\ 0 & \bar\chi_2 \end{pmatrix}.
\end{align}
Here, the ratio $\dfrac{\bar\chi_1}{\bar\chi_2}$ is not equal to the mod $p$-cyclotomic character or its inverse. 
\item\label{item:barrhoshaperamifieddisc} $\bar\rho \mid_{G_{K_{\ell_0}}}$ is unramified. Furthermore, there exists an isomorphism of $\overline{\mathbb{F}}_p$ representations:
\begin{align*}
\bar\rho \mid_{G_{K_{\ell_0}}} \cong \begin{pmatrix}
    \bar\chi_1 & \star & 0 & 0 \\ 0 & \bar\chi_2 & 0 & 0  \\ 0 & 0 & \bar\chi_1 & \star \\ 0 & 0 & 0 &  \bar\chi_2 
\end{pmatrix}.
\end{align*}
\item\label{item:barrhocyclicorder2} $\mathrm{Image}(\bar\rho({I_{\ell}}))$ is a cyclic group of order $2$ in $\Gl_4(\overline{\mathbb{F}}_p)$ with eigenvalues of the non-trivial element equalling $\{1,1,-1,-1\}$.

\end{enumerate}
\end{lemma}

\begin{proof}
The fact that $\bar\tau \mid_{G_{K_{\ell_0}}}$ is unramified follows from hypothesis \ref{hyp:squarefree}. The shape given in equation (\ref{eq:basischangebartaudeltak}) follows by utilizing the Jordan canonical form of $\bar\tau(\mathrm{Frob}_{\ell_0})$, for some lift of Frobenius at $\ell_0$. The fact that the ratio $\dfrac{\bar\chi_1}{\bar\chi_2}$ is not equal to the mod $p$-cyclotomic character or its inverse follows by using hypothesis \ref{hyp:congruence2} along with the equality $\det(\tau_0)(\Frob_{\ell_0})=\ell^{\kappa_1-1}$. Thus, we obtain part \ref{item:part1ramifieddisc} of the lemma. 

The extension (say $K'_{\ell_0}$) of $K_{\ell_0}$ cut out by $\bar\tau \mid_{G_{K_{\ell_0}}}$ is unramified. Since $\ell$ is ramified in the extension $K/\Q$, the extension $K'_{\ell_0}/\Q_{\ell}$ is abelian. Combining this observation along with the definition of $\bar\tau^c$ allows us to conclude that $\bar\tau \mid_{G_{K_{\ell_0}}} \cong \bar\tau^c \mid_{G_{K_{\ell_0}}}$. Thus, we obtain part \ref{item:barrhoshaperamifieddisc} of the lemma. 

Part (\ref{item:barrhocyclicorder2}) follows directly from the construction of $\bar\rho$ and the fact that $\bar\rho(I_{\ell_0})$ is trivial. 
\end{proof}

\begin{proposition}\label{generallemma:elldividingdiscK}
Suppose $\ell$ is a prime of $\Q$ dividing $\Delta_K$.  Let $R$ denote a complete local Noetherian $\mathrm{O}$-algebra with finite residue field $\mathbb{F}$.  Let $\rho:\Gal{\overline{\Q}}{\Q} \rightarrow \GSp_4(R)$ denote a lift of $\bar\rho$. 
\begin{enumerate}[label=(\roman*)]
    \item\label{item:partareducedoverOramifieddisc} Suppose $R$ is a reduced ring that is finite and torsion-free over $\Z_p$. Then, $\mathrm{Image}(\rho({I_{\ell}}))$ is a cyclic group of order $2$. Furthermore, the characteristic polynomial of the non-trivial element equals $(X-1)^2(X+1)^2$.
    \item\label{item:patchinggsp4discshape} Suppose that the conclusion of the previous part holds. Then, there exists a matrix $P_{\rho,\ell}$ in $\GSp_4(R)$ such that $\rho(I_{\ell})$ is a cyclic group generated by $P_{\rho,\ell}\begin{pmatrix} 1 & 0 & 0 & 0 \\ 0 & -1 & 0 & 0\\0 & 0 & -1 & 0 \\ 0 & 0 & 0 & 1\end{pmatrix}P_{\rho,\ell}^{-1}$.
\end{enumerate}
\end{proposition}

\begin{proof}
Since $R$ is reduced, to prove part \ref{item:partareducedoverOramifieddisc} of the proposition  without loss of generality we may assume that $R$ equals a subring $\mathrm{O}'$ of a finite extension of $\Q_p$.  Let $\ell_0$ denote the unique prime of $K$ lying above $\ell$. Let $g_{\ell_0}$ denote a lift in $I_{\ell_0}$ of the topological generator of the pro-$p$ part (which is isomorphic to $\Z_p$)  of the tame inertia group at $\ell_0$.  Just as in Claim \ref{claim:1isonlyeigenvalue} of Proposition \ref{generallemma:elldividinglevel}, we can deduce that the only eigenvalue of $\rho(g_{\ell_0})$ equals $1$. This deduction once again uses hypothesis \ref{hyp:congruence2} and the fact that $\bar\rho(\Frob_{\ell_0})$ has only two distinct eigenvalues. 

Let $X$ denote $\rho(\Frob_{\ell_0})$, for some lift of Frobenius at $\ell_0$. We need to rule out the following Jordan canonical forms for $\rho(g_{\ell_0})$: 
\begin{align*}
\begin{pmatrix} 1 & 1 & 0 & 0 \\ 0 & 1 & 1 & 0 \\ 0 & 0 & 1 & 1 \\ 0 & 0 & 0 & 1    \end{pmatrix}, \quad \text{or} \quad   \begin{pmatrix} 1 & 1 & 0 & 0 \\ 0 & 1 & 1 & 0 \\ 0 & 0 & 1 & 0 \\ 0 & 0 & 0 & 1    \end{pmatrix}, \quad \text{or} \quad \begin{pmatrix} 1 & 1 & 0 & 0 \\ 0 & 1 & 0 & 0 \\ 0 & 0 & 1 & 1 \\ 0 & 0 & 0 & 1    \end{pmatrix}, \quad \text{or} \quad \begin{pmatrix} 1 & 1 & 0 & 0 \\ 0 & 1 & 0 & 0 \\ 0 & 0 & 1 & 0\\ 0 & 0 & 0 & 1    \end{pmatrix}.
\end{align*}
Let $B$ denote the Jordan canonical form of $\rho(g_{\ell_0})$. One can solve for a general formula for $X$ so that the relation $XB = B^\ell X$ holds. In each of the cases, an explicit computation for this general formula of $X$ shows that we can find two eigenvalues (say $a,b$) of $X$ such that $a=\ell b$. Therefore, one obtains two eigenvalues of $\bar\rho(\Frob_{\ell_0})$ such that their ratio equals $\bar{\ell}$. This is prohibited by Lemma \ref{lemma:ramifieddiscshape} and hypothesis \ref{hyp:congruence2}. This allows us to conclude that $\rho(g_{\ell_0})$ equals $I_4$. Therefore, $\rho(I_\ell)$ factors through the quotient ${I_\ell/I_{\ell_0}}\cong{\Z/2\Z}$. Therefore, the eigenvalues of the non-trivial element of $\rho(I_\ell)$ equals $\pm 1$. Part \ref{item:partareducedoverOramifieddisc} now follows from Lemma \ref{lemma:ramifieddiscshape}(\ref{item:barrhocyclicorder2}). 

To deduce part \ref{item:patchinggsp4discshape}, part \ref{item:partareducedoverOramifieddisc} allows us to find an eigenbasis $\{\vec{e}_1,\vec{e}_2, \vec{e}_3, \vec{e}_4\}$ for $\rho$ so that  the non-trivial element of $\rho(I_\ell)$ -- call it $x$ -- acts by $\begin{pmatrix} 1 & 0 & 0 & 0 \\ 0 &-1 & 0 & 0 \\  0&0&-1 &0 \\ 0 & 0 & 0 & 1 \end{pmatrix}$ under this new basis. We now note that the restriction of the similitude character to $I_\ell$ is trivial, because residually, the restriction of  $\bar\lambda$ is trivial (since residually, it is a power of the cyclotomic character) and $p$ does not divide $\ell-1$. Hence, $\langle x \cdot e_i, x \cdot e_j  \rangle = \langle e_i,e_j\rangle$, for all $i,j\in \left\{1,2,3,4\right\}$. Therefore, 
\begin{align*}
\langle \vec{e}_1,\vec{e}_2\rangle = \langle \vec{e}_1,\vec{e}_3\rangle = \langle \vec{e}_2,\vec{e}_4\rangle = \langle \vec{e}_3,\vec{e}_4\rangle = 0. 
\end{align*}
Appropriately scaling the elements in $\{\vec{e}_1,\vec{e}_2, \vec{e}_3, \vec{e}_4\}$, we can have  $\langle \vec{e}_1,\vec{e}_4\rangle = \langle \vec{e}_2,\vec{e}_3\rangle = 1$. We let $P_{\rho,\ell}^{-1}$ denote the change of basis matrix. One sees that $(P_{\rho,\ell}^{T})^{-1} J P^{-1}_{\rho,\ell} = J$, thereby proving the proposition. 
\end{proof}

\section{Proving $R=\mathbb{T}$} \label{subsec:RequalsT}
\subsection{Minimal deformation problems}
We define various minimal deformation rings associated to $\bar\rho$. We follow the exposition given in Pilloni's work (\cite[Sections 5.3 and 5.4]{MR2920881}). Let $\mathcal{C}$ denote the category consisting of complete Artinian local $\mathrm{O}$-algebras with finite residue field $\mathbb{F}$. Here, $\mathrm{O}$ denotes the Witt vectors of $\mathbb{F}$.

\begin{definition}
\label{def:orddef}
    Consider the functor $\mathbf{D}: \mathcal{C} \rightarrow \mathrm{Sets}$, that sends an object $R$ of $\mathcal{C}$ to the set of isomorphism classes (over $\GSp_4$) of lifts $\rho : G_{\Q,S} \to \GSp_4(R)$ of $\bar\rho$ satisfying the following properties:
    \begin{enumerate}
        \item  There exists a matrix $P_{\rho,p} \in \GSp_4(R)$ such that 
        \[\rho(g) = P_{\rho,p} \begin{pmatrix} \chi_1 & * & * & * \\ 0 & \chi_2 & * & *\\0 & 0 & \chi_3 & * \\ 0 & 0 & 0 & \chi_4\end{pmatrix}P_{\rho,p}^{-1}, \qquad \forall \ g \in G_{\Q_p}.\]
          Here, for each $i \in \{1,2,3,4\}$, we have a character $\chi_i:G_{\Q_p} \rightarrow R^\times$. Furthermore, we impose the condition that the character $\chi_4$ is unramified. In this case, we say $\rho$ is \textit{ordinary at $p$}.
         \item If $\ell$ divides $M$, then there exists a $P_{\rho,\ell} \in \GSp_4(R)$ such that $\rho(I_{\ell})$ is the cyclic group generated by $P_{\rho,\ell}\begin{pmatrix} 1 & 1 & 0 & 0 \\ 0 & 1 & 0 & 0\\0 & 0 & 1 & -1 \\ 0 & 0 & 0 & 1\end{pmatrix}P_{\rho,\ell}^{-1}$.  In Genestier--Tilouine's work \cite{MR2234862}, these primes are classified as \textit{type $\mathrm{UN}_{2,2}$}. In Pilloni's work \cite{MR2920881}, these primes are classified as \textit{type $\epsilon_2$}.
         \item If $\ell$ ramifies in $K/\Q$, then there exists a $P_{\rho,\ell} \in \GSp_4(R)$ such that $\rho(I_{\ell})$ is the cyclic group generated by $P_{\rho,\ell}\begin{pmatrix} 1 & 0 & 0 & 0 \\ 0 & -1 & 0 & 0\\0 & 0 & -1 & 0 \\ 0 & 0 & 0 & 1\end{pmatrix}P_{\rho,\ell}^{-1}$. 
         The deformation condition is similar to the one given for primes classified as \textit{type $\mathrm{PR}_{2}$} in Genestier--Tilouine's work \cite{MR2234862}. In Pilloni's work \cite{MR2920881}, those primes are classified as \textit{type $\chi_\ell$}. The main difference between their conditions and ours is that in our setting, the $+1$ and $-1$ eigenspaces do not form isotropic subspaces.  
    \end{enumerate}

\end{definition}

\begin{remark}
Strictly speaking, if $2$ ramifies in $K/\Q$, then it is not included as a prime of \textit{type $\mathrm{PR}_{2}$} or  \textit{type $\chi_\ell$} in the above works of Genestier--Tilouine and Pilloni respectively. However, their calculations would follow through verbatim since we are working with the minimal deformation problem. Thus, we do not consider this case separately. 
\end{remark}

    It follows, from \cite[Proposition 2.1]{MR2264659}, that the functor $\mathbf{D}$ is pro-representable by a complete local Noetherian $\mathrm{O}$-algebra with finite residue field $\mathbb{F}$. We denote this universal deformation ring by $R^{\min}$ and call it the minimal ordinary deformation ring of $\bar\rho$. 

\begin{definition} \label{defn:primesTW}
    A prime number $q$ is called a \textit{Taylor-Wiles prime} if it satisfies the following conditions:
    \begin{itemize}
        \item  $q$ does not divide $M\Delta_K$,
        \item $q \equiv 1 \pmod{p}$, and
        \item $\bar\rho(\text{Frob}_q)$ has four distinct eigenvalues.
    \end{itemize}
    If $q$ is a Taylor-Wiles prime, then we fix an ordering $\bar\alpha_{1,q}, \bar\alpha_{2,q}, \bar\alpha_{3,q}, \bar\alpha_{4,q}$ of the eigenvalues of $\bar\rho(\text{Frob}_q)$ such that $\bar\alpha_{1,q}\bar\alpha_{4,q} = \bar\alpha_{2,q}\bar\alpha_{3,q}$.
    Thus, in this case, $$\bar\rho|_{G_{\Q_q}} \simeq \bar\chi_{1,q} \oplus \bar\chi_{2,q} \oplus \bar\chi_{3,q} \oplus \bar\chi_{4,q}, $$
    where, for all $1 \leq i \leq 4$, we have $\bar\chi_{i,q}$ is an unramified character with $\bar\chi_{i,q}(\text{Frob}_q) = \bar\alpha_{i,q}$. It will be possible  to fix an ordering of these characters $\bar\chi_{1,q} , \bar\chi_{2,q}, \bar\chi_{3,q},\bar\chi_{4,q}$ such that $\bar\chi_{1,q}(\text{Frob}_q) \bar\chi_{2,q}(\text{Frob}_q)$ does not belong to the set $\{\bar\chi_{1,q}(\text{Frob}_q) \bar\chi_{3,q}(\text{Frob}_q), \bar\chi_{2,q}(\text{Frob}_q) \bar\chi_{4,q}(\text{Frob}_q),\bar\chi_{3,q}(\text{Frob}_q) \bar\chi_{4,q}(\text{Frob}_q)\}$. Henceforth, for the rest of the article, we fix one such ordering.
\end{definition}
    
Henceforth, we implicitly increase our residue field so that all (distinct) roots of the reductions of the Hecke polynomials at Taylor--Wiles primes lie in $\mathbb{F}$. This amounts to changing our coefficient ring by a finite extension, which does not alter the conclusion of our results. Given a (possibly empty) finite set $Q$ of Taylor-Wiles primes, we define the following deformation problem associated to $Q$ with finite weights.

    \begin{definition}
    \label{defi:Qdef}
Let $k'_1, k'_2$ be positive integers such that $k'_i \equiv k_i \pmod{p-1}$ for $i=1,2$ and $Q$ be a finite set of Taylor-Wiles primes.
    Let $\mathbf{D}_{Q,(k'_1,k'_2)}$ be the functor from $\mathcal{C}$ to $\mathrm{Sets}$ which sends an object $R$ of $\mathcal{C}$ to the set of isomorphism classes (over $\GSp_4$) of lifts $\rho : G_{\Q,S \cup Q} \to \GSp_4(R)$ of $\bar\rho$ satisfying the following properties:
    \begin{enumerate}
  \item   There exists a matrix $P_{p} \in \GSp_4(R)$ such that 
        \[\rho(g) = P_{p} \begin{pmatrix} \chi_1 & * & * & * \\ 0 & \chi_2 & * & *\\0 & 0 & \chi_3 & * \\ 0 & 0 & 0 & \chi_4\end{pmatrix}P_{p}^{-1}, \qquad \forall \ g \in G_{\Q_p},\] such that the restriction to inertia has the following shape:
         \[\rho(g) = P_{p} \begin{pmatrix} \epsilon^{k'_1+k'_2-3} & * & * & * \\ 0 & \epsilon^{k'_1 -1}& * & *\\0 & 0 & \epsilon^{k'_2-2} & * \\ 0 & 0 & 0 & 1\end{pmatrix}P_{p}^{-1}, \qquad \forall \ g \in I_p,\]
          Here, for each $i \in \{1,2,3,4\}$, we have a character $\chi_i:G_{\Q_p} \rightarrow R^\times$.

         \item If $\ell$ divides $M$, then there exists a $P_{\ell} \in \GSp_4(R)$ such that $\rho(I_{\ell})$ is the cyclic group generated by $P_{\ell}\begin{pmatrix} 1 & 1 & 0 & 0 \\ 0 & 1 & 0 & 0\\0 & 0 & 1 & -1 \\ 0 & 0 & 0 & 1\end{pmatrix}P_{\ell}^{-1}$.
         \item If $\ell$ ramifies in $K$, then there exists a $P_\ell \in \GSp_4(R)$ such that $\rho(I_{\ell})$ is the cyclic group generated by $P_{\ell}\begin{pmatrix} 1 & 0 & 0 & 0 \\ 0 & -1 & 0 & 0\\0 & 0 & -1 & 0 \\ 0 & 0 & 0 & 1\end{pmatrix}P_{\ell}^{-1}$.
        
         \item If $q \in Q$, then there exists a $P_q \in \GSp_4(R)$, such that for all $g \in G_{\Q_q}$, we have  
         \[\rho|_{G_{\Q_q}}(g) = P_q \begin{pmatrix}  \chi_{1,q}(g)  & 0 & 0 & 0 \\ 0 & \chi_{2,q}(g)  &  0 & 0 \\ 0 & 0 & \chi_{3,q}(g) & 0 \\ 0 & 0 & 0 & \chi_{4,q}(g)  \end{pmatrix} P_q^{-1}.\] Here, $\chi_{i,q} : G_{\Q_q} \to R^{\times}$ is a character lifting $\bar\chi_{i,q}$ for every $1 \leq i \leq 4$. Furthermore, the $G_{\Q_q}$-characters $\chi_{1,q}$ and $\chi_{2,q}$ are unramified with $\chi_{2,q}\chi_{3,q}= \chi_{1,q}\chi_{4,q}$. In Genestier--Tilouine's work, these primes are classified as  \textit{type $\mathcal{D}_q$}. (We note that by \cite[Lemma 5.1.1]{MR2234862} this condition is equivalent to the ones considered by \cite{MR2234862} and \cite{MR2920881} for Taylor--Wiles primes.)
    \end{enumerate}
  We call an element of $\mathbf{D}_{\emptyset,(k'_1,k'_2)}(R)$ a \textit{minimal deformation of $\bar\rho$ of weight $(k'_1,k'_2)$}.
\end{definition}
  It follows, from \cite[Proposition 5.2.1]{MR2234862}, that the functor $\mathbf{D}_{Q,(k'_1,k'_2)}$ is pro-representable. We denote this universal deformation ring by $R^{\min}_{Q,(k'_1,k'_2)}$. We let  $R^{\min}_{(k'_1,k'_2)}$ denote $R^{\min}_{\emptyset,(k'_1,k'_2)}$. 
  
Note that the definition of $R^{\min}_{Q,(k'_1,k'_2)}$ depends on the ordering of eigenvalues of $\bar\rho(\text{Frob}_q)$ that we fixed for every $q \in Q$. Combining \cite[Proposition 5.6]{MR2920881} and Lemmas~\ref{lemma:TWH1}, \ref{lemma:TWH2}, \ref{lemma:TWH3} that are proved below, we obtain the existence of suitable sets of Taylor-Wiles primes.

    \begin{proposition}[Pilloni]\label{prop:pilloni}
        Let $k'_1, k'_2$ be positive integers such that $k'_i \equiv k_i \pmod{p-1}$ for $i=1,2$.
        Then for every integer $m \geq 1$, there exists a set of Taylor-Wiles primes $Q_m$ such that
        \begin{enumerate}
            \item If $q \in Q_m$, then $q \equiv 1 \pmod{p^m}$,
            \item $|Q_m| =r$ is independent of $m$,
            \item $R^{\min}_{Q_m,(k'_1,k'_2)}$ is a quotient of $\mathcal{O} \llbracket X_1,\cdots,X_r \rrbracket$.
        \end{enumerate}
    \end{proposition}

\subsection{Verifying the hypotheses for existence of Taylor--Wiles primes}

 In \cite{MR2920881}, Pilloni gives sufficient hypotheses for the existence of a suitable collection of Taylor-Wiles primes alluded to in Proposition \ref{prop:pilloni}. These hypotheses are labelled (H1), (H2) and (H3). These hypotheses are verified in Lemmas \ref{lemma:TWH1}, \ref{lemma:TWH2} and \ref{lemma:TWH3}. To state these hypotheses, we need to work with the Lie algebra $\mathfrak{g}$ of $\GSp_4(\mathbb{F})$. Note that we have a direct sum decomposition $\mathfrak{g} = \mathfrak{g}^0 \oplus \mathfrak{c}$, where $\mathfrak{c}$ is the subspace of $M_4(\mathbb{F})$ consisting of scalar matrices and $\mathfrak{g}^0$ consists of all matrices $X$ in $M_4(\mathbb{F})$ satisfying $X^TJ + JX = 0$.  Concretely, $\mathfrak{g}^0$ is the subspace of all matrices $M \in M_4(\mathbb{F})$ of the form $M = \begin{pmatrix} a_1 & b_1 & c_1 & d_1 \\ a_2 & b_2 & c_2 & c_1 \\a_3 & b_3 & -b_2 & -b_1 \\a_4 & a_3 & -a_2 & -a_1 \end{pmatrix}$. We get $11$-dimensional and $10$-dimensional $G_{\Q,S}$-representations respectively
\begin{align*}
    \mathrm{Ad}(\bar\rho):G_{\Q,S} \rightarrow \mathrm{Aut}(\mathfrak{g}) \cong \Gl_{11}(\mathbb{F}), \qquad     \mathrm{Ad}^0(\bar\rho):G_{\Q,S} \rightarrow \mathrm{Aut}(\mathfrak{g}^0) \cong \Gl_{10}(\mathbb{F}),
\end{align*}
by the conjugation via $\bar\rho$. Note that $\mathrm{Ad}(\bar\rho) \cong \mathrm{Ad}^0(\bar\rho) \oplus \mathbbm{1}$. 

Let $M^0_2(\mathbb{F})$ denote the subspace of $M_2(\mathbb{F})$ consisting of matrices with trace $0$. We also get $4$-dimensional and $3$-dimensional $G_{K,S}$-representations by the conjugation action of $\bar\tau$:
\begin{align*}
 \mathrm{Ad}(\bar\tau):G_{K,S} \rightarrow \mathrm{Aut}(M_2(\mathbb{F})) \cong \Gl_{4}(\mathbb{F}), \qquad     \mathrm{Ad}^0(\bar\tau):G_{K,S} \rightarrow \mathrm{Aut}(M^0_2(\mathbb{F})) \cong \Gl_{3}(\mathbb{F}).
\end{align*}
We define $\mathrm{Ad}(\bar\tau^c)$  and $\mathrm{Ad}^0(\bar\tau^c)$ similarly by replacing $\bar\tau$ by $\bar\tau^c$ above.

\begin{lemma}[Hypothesis H1]
\label{lemma:TWH1}
    The extension of $\Q$ fixed by $\ker(\mathrm{Ad}^0(\bar\rho))$ does not contain $\zeta_p$. 
\end{lemma}
\begin{proof}
Let $L'$ be the extension of $\Q$ fixed by $\ker(\mathrm{Ad}^0(\bar\rho))$. By observing that the centralizer of $\mathfrak{g}^0$ in $M_4(\mathbb{F})$ consists of the scalar matrices,  we have a natural isomorphism 
\begin{align} \label{eq:isoprojad0}
    \mathrm{ProjIm}(\bar\rho) \xrightarrow {\cong} \mathrm{Im}(\mathrm{Ad}^0(\bar\rho))
\end{align} Here, $\mathrm{ProjIm}(\bar\rho)$ denotes the image of $\bar\rho$ in $\mathrm{PGL}_4(\mathbb{F})$. Using the isomorphism in equation (\ref{eq:isoprojad0}), one sees that any element of $G_{L',S}$ is scalar under $\bar\rho$, whereas by the construction in the previous section, none of the elements of $G_{\Q,S}-G_{K,S}$ maps to a scalar under $\bar\rho$. Thus, $K \subset L'$. Thus, to prove the lemma, it suffices to prove that $K(\zeta_p) \not\subset L'$. Since $p$ splits in $K$, we have $\Gal{K(\zeta_p)}{K}$ is cyclic of order $p-1$. Thus, to prove the lemma, it suffices to show that the abelian quotients of  $\mathrm{ProjIm}(\bar\rho(G_{K,S}))$ have exponent dividing $2$. \\

The big image hypothesis along with a theorem of Dickson \cite[Theorem 3.7]{MR2172950} tells us that $\mathrm{ProjIm(\bar\tau)}$ is isomorphic to $\mathrm{PGL}_2(\mathbb{F}')$ or $\mathrm{PSL}_2(\mathbb{F}')$, for some finite field $\mathbb{F}' \subset \mathbb{F}$. A group theory lemma \cite[Lemma 3.10]{MR2172950} would further let us conclude that $\mathrm{Im}(\bar\tau) \supset \mathrm{SL}_2(\mathbb{F}')$. Furthermore, we claim that \begin{align}\label{claim:inclusionsl2timessl2}
\bar\rho(G_{K,S}) \supset \mathrm{SL}_2(\mathbb{F}') \times \mathrm{SL}_2(\mathbb{F}').
\end{align}
 To see this, recall that if $h \in G_{K,S}$, then $$\bar\rho(h) = \begin{pmatrix} a & 0 & 0 & b\\ 0 & a' & b' & 0\\ 0 & c' & d' & 0 \\ c & 0 & 0 & d\end{pmatrix} \text{ with } ad-bc = a'd'-b'c' \text{ and } \begin{pmatrix} a & b \\ c & d\end{pmatrix}, \begin{pmatrix} a' & b' \\ c' & d'\end{pmatrix} \in \mathbb{F}^\times\mathrm{GL}_2(\mathbb{F}').$$

Let $\mathcal{N}_1$ denote the set of all elements $g \in \mathrm{SL}_2(\mathbb{F}')$ such that $(g,1)$ belongs to $\mathrm{Im}\left(\bar\rho\mid_{G_{K,S}}\right) $ via the embedding above. Similarly, let $\mathcal{N}_2$ denote the set of all elements $g \in \mathrm{SL}_2(\mathbb{F}')$ such that $(1,g)$ belongs to $\mathrm{Im}\left(\bar\rho\mid_{G_{K,S}}\right) $ via the embedding above. To prove the claim in equation (\ref{claim:inclusionsl2timessl2}), it suffices to show that $\mathcal{N}_1=\mathcal{N}_2 =\mathrm{SL}_2(\mathbb{F}')$. Note that $\mathcal{N}_1$ and $\mathcal{N}_2$ are normal subgroups of $\mathrm{SL}_2(\mathbb{F}')$ since $\mathrm{SL}_2(\mathbb{F}')$ is contained in both $\mathrm{Im}(\bar\tau)$ and $\mathrm{Im}(\bar\tau^c)$. Since $p \geq 5$, our big image hypothesis \ref{hyp:bigimage} now gives us the desired claim as the only nontrivial proper normal subgroup of $\mathrm{SL}_2(\mathbb{F}')$ is $\{\pm I_2\}$. Thus, we get inclusions $\mathrm{SL}_2(\mathbb{F}') \times \mathrm{SL}_2(\mathbb{F}') \subset \bar\rho(G_{K,S}) \subset H$, where 
\[H= \left\{(\mu_1A,\mu_2B) \mid \mu_1,\mu_2 \in \mathbb{F}^\times, \ A, B \in \mathrm{GL}_2(\mathbb{F}'), \ \mu_1^2\det(A) = \mu_2^2\det(B)  \in (\mathbb{F}_p^\times)^{\kappa_1-1} \right\}. \]
 Let $\mathcal{G}$ denote the natural image of $\mathrm{SL}_2(\mathbb{F}') \times \mathrm{SL}_2(\mathbb{F}')$ inside $\mathrm{PGL}_4(\mathbb{F})$. Similarly, let $H'$ denote the natural image of $H$ inside $\mathrm{PGL}_4(\mathbb{F})$. We have natural inclusions 
   \begin{align*}
       \mathcal{G} \subset \mathrm{ProjIm}(\bar\rho(G_{K,S})) \subset H'.
   \end{align*}
   Consider an element $(\mu_1A,\mu_2B)$ in $H$. Observe that $\mu_1^2, \mu_2^2 \in \mathbb{F}'$. Consider the element $\mu_1^2\det(A)$ (say $c$) in $\mathbb{F}'$. Then, we note that $\dfrac{1}{c}(\mu_1^2A^2,\mu_2^2B^2) \in \mathrm{SL}_2(\mathbb{F}') \times \mathrm{SL}_2(\mathbb{F}')$. This observation tells us that the group $ H'/ \mathcal{G}$  has exponent $2$. Since $\mathrm{SL}_2(\mathbb{F}')$ is its own commutator, we get that any abelian quotient of $\mathrm{ProjIm}(\bar\rho(G_{K,S}))$ has exponent dividing $2$. This concludes the proof of the lemma. 
 \end{proof}

\begin{lemma}[Hypothesis H2]\label{lemma:TWH2}
For every $m$, there exists an element $\sigma_m \in G_{\Q(\zeta_{p^m}),S}$ such that $\bar\rho(\sigma_m)$ has $4$ distinct eigenvalues and each irreducible subrepresentation of $\text{Ad}(\bar\rho)|_{G_{\Q(\zeta_{p^m}),S}}$ contains an eigenvector of the action of $\sigma_m$ with eigenvalue $1$.
\end{lemma}

\begin{proof}
    We follow the argument given in \cite[Example 4.11]{MR4079417}.
Observe that we have a direct sum decomposition of $\mathbb{F}[G_{\Q,S}]$-modules:
    \begin{align}\label{eq:liealgdirectsum}\mathfrak{g} \cong \bigg(\mathrm{Ind}^{G_{\Q,S}}_{G_{K,S}}\mathrm{Ad}^{0}(\bar\tau) \bigg) \oplus \bigg(\bar\tau \otimes \bar\tau^c(\det(\bar\tau)^{-1})\bigg) \oplus \mathbb{F}.\end{align}

The big image hypothesis \ref{hyp:bigimage} lets us conclude that $\mathrm{Ad}^{0}(\bar\tau)$ and $\mathrm{Ad}^{0}(\bar\tau^c)$ are non-isomorphic absolutely irreducible $G_{K,S}$-representations. Since neither of them are stable under the action of $\Gal{K}{\Q}$, the induction $\mathrm{Ind}^{G_{\Q,S}}_{G_{K,S}}\mathrm{Ad}^{0}(\bar\tau)$ is an irreducible $G_{\Q,S}$-representation. 

The construction of $\bar\rho$ in \S \ref{sec:galoisrep} gives us an order two element $\sigma := \begin{pmatrix} 0 & 1 & 0 & 0\\ 1 & 0 & 0 & 0\\ 0 & 0 & 0 & 1 \\ 0 & 0 & 1 & 0 \end{pmatrix}$ in the image of the inertia group $I_{\ell_0}$ (where the rational prime $\ell_0$ was chosen to be ramified in $K/\Q$). Since $\Q(\zeta_{p^m})/\Q$ is unramified at $\ell_0$, note that $\bar\rho({G_{\Q(\zeta_{p^m}),S}})$ contains the order two element $\sigma$ for all $m \geq 0$.

Note that $\Gal{\Q(\zeta_{p^m})}{\Q}$ is a cyclic abelian group of order $p^{m-1}(p-1)$. Therefore, the quotient of the group $\bar\rho({G_{K,S}})$ by $\bar\rho({G_{K(\zeta_{p^m}),S}})$ is abelian. Since the commutator subgroup of $\mathrm{SL}_2(\mathbb{F}') \times \mathrm{SL}_2(\mathbb{F}')$ equals itself (where the field $\mathbb{F}'$ comes in the proof of Lemma \ref{lemma:TWH1}), we see that $\mathrm{SL}_2(\mathbb{F}') \times \mathrm{SL}_2(\mathbb{F}')$ is contained in $\bar\rho\left({G_{K(\zeta_{p^m}),S}}\right)$ for all $m \geq 0$.

The observations in the above three paragraphs let us conclude that $\mathrm{Ind}^{G_{\Q,S}}_{G_{K,S}}\mathrm{Ad}^{0}(\bar\tau)$ is an irreducible $G_{\Q(\zeta_{p^m}),S}$-representation. Furthermore, an explicit computation using the fact that $\mathrm{SL}_2(\mathbb{F}') \times \mathrm{SL}_2(\mathbb{F}')$ is a subgroup of $\bar\rho({G_{K(\zeta_{p^m}),S}})$ would tell us that $\bigg(\bar\tau \otimes \bar\tau^c(\det(\bar\tau)^{-1})\bigg)$ is an irreducible $G_{K(\zeta_{p^m}),S}$-representation. 
Indeed, fix a basis $\{w_1,w_2\}$ of $\bar\tau^c \otimes (\det(\bar\tau)^{-1})$ and note that every element of $\bar\tau \otimes \bar\tau^c(\det(\bar\tau)^{-1})$ can be written as $v_1 \otimes w_1 + v_2 \otimes w_2$ with $v_1, v_2 \in \bar\tau$. Let $g_1, g_2 \in G_{K(\zeta_{p^m}),S}$ be such that \[\bar\tau(g_1) =\bar\tau(g_2) = \begin{pmatrix} 1 & 0 \\ 0 & 1\end{pmatrix} \quad \bar\tau^c(g_1) = \begin{pmatrix} 1 & 0 \\ 1 & 1\end{pmatrix}, \bar\tau^c(g_2) = \begin{pmatrix} 0 & 1 \\ -1 & 0\end{pmatrix}.\]
Suppose $\mathcal{V}$ is a $G_{K(\zeta_{p^m}),S}$-subrepresentation of $\bar\tau \otimes \bar\tau^c(\det(\bar\tau)^{-1})$. Let $v := v_1 \otimes w_1 + v_2 \otimes w_2 \in \mathcal{V}$ be a non-zero element. After replacing $v$ by $v+g_2.v$ if necessary, we can assume $v_1 \neq 0$ and $v_2 \neq 0$. Thus, $g_1.v-v=v_1 \otimes w_2 \in \mathcal{V}$. Thus, \[S =\{g.(v_1 \otimes w_2) \mid g \in G_{K(\zeta_{p^m}),S} \text{ with } \bar\tau^c(\det(\bar\tau)^{-1})(g) = \text{Id}\} =\{v' \otimes w_2 \mid v' \in \bar\tau\} \subset \mathcal{V}.\]
Now $g_2.S =\{g_2.s \mid s \in S\} = \{v' \otimes w_1 \mid v' \in \bar\tau\} \subset \mathcal{V}$. Thus it follows that $\mathcal{V}=\bar\tau \otimes \bar\tau^c(\det(\bar\tau)^{-1})$ which proves the claim.

 Using the explicit model for the Lie algebra $\mathfrak{g}$,  we may realize the direct sum decomposition of $\mathbb{F}[G_{\Q,S}]$-modules given in equation (\ref{eq:liealgdirectsum}) as follows:
    \[\mathfrak{g} \simeq W \oplus V \oplus \mathfrak{c},\] where
    \begin{itemize}
    \item $W$ is the subspace of $\mathfrak{g}$ consisting of matrices of the form $\begin{pmatrix} a_1 & 0 & 0 & d_1 \\ 0 & b_2 & c_2 & 0 \\0 & b_3 & -b_2 & 0 \\a_4 & 0 & 0 & -a_1 \end{pmatrix}$,
    \item $V$ is the subspace of $\mathfrak{g}$ consisting of matrices of the form $\begin{pmatrix} 0 & b_1 & c_1 & 0 \\ a_2 & 0 & 0 & c_1 \\a_3 & 0 & 0 & -b_1 \\0 & a_3 & -a_2 & 0 \end{pmatrix}$.
    \item $\mathfrak{c}$ is the subspace of scalar matrices. 
    \end{itemize}

     Consider 
     \begin{align*}
     h_0 := \begin{pmatrix} 0 & 1 & 0 & 0\\ 2 & 0 & 0 & 0\\ 0 & 0 & 0 & 1/2 \\ 0 & 0 & 1 & 0 \end{pmatrix} = \underbrace{\begin{pmatrix} 1 & 0 & 0 & 0\\ 0 & 2 & 0 & 0\\ 0 & 0 & 1/2 & 0\\ 0 & 0 & 0 & 1 \end{pmatrix}}_{\in \mathrm{SL}_2(\mathbb{F}') \times  \mathrm{SL}_2(\mathbb{F}')} \sigma
     \end{align*}
    Clearly $h_0$ belongs to $\bar\rho(G_{\Q(\zeta_{p^m}),S})$, for all $m \geq 0$. Let $h_0=\bar\rho(\sigma_m)$, for some element $\sigma_m \in G_{\Q(\zeta_{p^m}),S}$.  Eigenvalues of $h_0$ are $\pm\sqrt{2}$ and $\pm1/\sqrt{2}$ (and hence distinct). Now \[\begin{pmatrix} 1 & 0 & 0 & 0\\ 0 & 1 & 0 & 0\\ 0 & 0 & -1 & 0 \\ 0 & 0 & 0 & -1 \end{pmatrix} \in W \text{ and } \begin{pmatrix} 0 & 1 & 0 & 0\\ 2 & 0 & 0 & 0\\ 0 & 0 & 0 & -1 \\ 0 & 0 & -2 & 0 \end{pmatrix} \in V\] are eigenvectors for $\sigma$  with eigenvalue $1$ (for the conjugation via $h_0$ action).
   Thus, $\sigma_m$ is an element satisfying  the requirements of the lemma.
\end{proof}
\begin{lemma}[Hypothesis H3]\label{lemma:TWH3}
$\dim(H^0(G_{\Q,S}, \mathfrak{g}))=1$ and $H^0(G_{\Q,S}, \mathfrak{g}(1)) = 0$.
\end{lemma}
\begin{proof}
   Since the $G_{\Q,S}$-subrepresentations $W$, $V$ and $\mathfrak{c}$ of $\mathfrak{g}$ --- introduced in the proof of Lemma \ref{lemma:TWH2} --- are all irreducible representations of $G_{\Q(\zeta_{p^m}),S}$, they are irreducible representations of $G_{\Q,S}$ as well.
   Both $W$ and $V$ have dimensions greater than $1$. Hence, neither of them contain a sub-representation with dimension $1$.
      As $\mathfrak{g} = W \oplus V \oplus \mathfrak{c}$, we have
$$\dim(H^0(G_{\Q,S}, \mathfrak{g}))= \dim(H^0(G_{\Q,S}, \mathfrak{c}))=1 \text{ and } H^0(G_{\Q,S}, \mathfrak{g}(1)) = H^0(G_{\Q,S}, \mathfrak{c}(1)) =0.$$
     This proves the lemma.
\end{proof}

\subsection{Auxiliary Klingen levels}
The Klingen subgroups $\mathrm{Kli}^{(0)}(\Z_\ell)$ and $\mathrm{Kli}^{(1)}(\Z_\ell)$ consist of matrices $M$ in $\GSp_4(\Z_\ell)$ having the following congruence shapes: 
\begin{align*}
M = \left[\begin{array}{cccc} \star & \star & \star & \star \\ 0 & \star  & \star & \star  \\ 0 & \star  & \star & \star \\ 0 & 0  & 0 & \star  \end{array} \right] (\mathrm{mod} \ \ell), \qquad  M = \left[\begin{array}{cccc} 1 & \star & \star & \star \\ 0 & \star  & \star & \star  \\ 0 & \star  & \star & \star \\ 0 & 0  & 0 & \star \end{array} \right] (\mathrm{mod} \ \ell).
\end{align*}

Let $* \in \{0,1\}$. We will also consider the commutative subalgebra $\Z[U_{\ell,0}, U_{\ell,1},U_{\ell,2}]$ of the $\Z$-algebra $C_c^\infty(\GSp_4(\Q_\ell)//\mathrm{Kli}^{(*)}_\ell(\Z_\ell),\Z)$ consisting of all $\mathrm{Kli}^{(*)}_\ell$ bi-invariant continuous functions $\GSp_4(\Q_\ell) \rightarrow \overline{\Q}$ with compact support, where $U_{\ell,2},U_{\ell,1}$ and $U_{\ell,0}$ are the characteristic functions of the  double cosets:

\begin{align*}
\mathrm{Kli}^{(*)}_\ell(\Z_\ell)\left[\begin{array}{cccc}1 \\ & 1\\ & &\ell\\ & & & \ell \end{array} \right] \mathrm{Kli}^{(*)}_\ell(\Z_\ell),  \quad  \mathrm{Kli}^{(*)}_\ell(\Z_\ell)\left[\begin{array}{cccc} 1 \\ & \ell \\ & & \ell \\ & & & \ell^2 \end{array} \right] \mathrm{Kli}^{(*)}_\ell(\Z_\ell), \quad \ell\mathrm{Kli}^{(*)}_\ell(\Z_\ell).
\end{align*}

Proposition \ref{prop:pilloni} guarantees the existence of a natural number $r$, such that for each $m \geq 1$, we have a finite set $Q_m$ of Taylor--Wiles primes with cardinality $r$ and further satisfying the conditions listed there.  We consider the following levels:
\begin{align*}
\Gamma_{\Iw_p,Q_m^{(1)}} &\coloneqq  \left( \Iw_p(\Z_p) \times  \prod_{q \in Q_m} \mathrm{Kli}^{(1)}(\Z_q) \times \prod_{\ell \notin Q_m \cup \{p\}} \Gamma_0^{(2)}(M\Z_\ell) \right)\cap \GSp_4(\Q), \\ 
\Gamma_{\Iw_p,Q_m^{(0)}} &\coloneqq  \left( \Iw_p(\Z_p) \times  \prod_{q \in Q_m} \mathrm{Kli}^{(0)}(\Z_q) \times \prod_{\ell \notin Q_m \cup \{p\}} \Gamma_0^{(2)}(M\Z_\ell) \right)\cap \GSp_4(\Q), \\
\Gamma_{\Iw_p} &\coloneqq  \left( \Iw_p(\Z_p)   \times \prod_{\ell \notin \{p\}} \Gamma_0^{(2)}(M\Z_\ell) \right)\cap \GSp_4(\Q).
\end{align*}
We consider modules of $\mathrm{O}$-valued $p$-ordinary Siegel cuspforms $S^{\ord}_{(k'_1,k'_2)}(\Gamma_{\Iw_p,Q_m^{(1)}}, \mathrm{O})$ and $S^\ord_{(k'_1,k'_2)}(\Gamma_{\Iw_p,Q_m^{(0)}}, \mathrm{O})$. We consider the corresponding Hecke algebras $H^{(1)}_{Q_m,(k'_1,k'_2)}$ and $H^{(0)}_{Q_m,(k'_1,k'_2)}$ respectively generated over $\mathrm{O}$ by the Hecke operators $U_{p,2},U_{p,1}$ along with $U_{q,2},U_{q,1}$ for all primes $q$ in $Q_m$ and $T_{\ell,2},T_{\ell,1},T_{\ell,0}$
for all primes $\ell \notin Q_m \cup \{p\}$. Similarly, we can also define the module of $\mathrm{O}$-valued $p$-ordinary Siegel cuspforms $S^{\ord}_{(k'_1,k'_2)}(\Gamma_{\Iw_p}, \mathrm{O})$ and the corresponding Hecke algebra $H_{(k'_1,k'_2)}$ generated over $\mathrm{O}$ by the Hecke operators $U_{p,2},U_{p,1}$  and $T_{\ell,2},T_{\ell,1},T_{\ell,0}$
for all primes $\ell \notin \{p\}$. \\

Consider weights $k'_1 \geq k'_2 \geq 3$ such that $(k'_1,k'_2) \equiv (a_0,b_0) \pmod {(p-1)\Z^2}$. Pick a $p$-ordinary Siegel cuspidal eigenform $f'$ in $S^{\ord}_{(k'_1,k'_2)}(\Gamma_{\Iw_p}, \mathrm{O})$ lifting $\bar\rho$ (for instance, we may choose $f'$ to be the $p$-ordinary $p$-stabilization of an appropriate stable Yoshida lift). For each $q \in Q_m$, we can find roots $\alpha_q,\beta_q$ of the Hecke polynomial at $q$ for $f'$ that lift $\bar\chi_1(\Frob_q)$, $\bar\chi_2(\Frob_q)$ as in Definition \ref{defi:Qdef}. Let $f'_{\mathrm{stab}}$ be the form in $S^{\ord}_{(k'_1,k'_2)}(\Gamma_{\Iw_p,Q_m^{(1)}}, \mathrm{O})$ obtained by $\prod_{q \in Q_m} q$-stabilization of $f'$ with respect to this choice. That is, for each $q \in Q_m$, the $qU_{q,1}$ and $U_{q,2}$ eigenvalues of $f'_{\mathrm{stab}}$ are $\alpha_q\beta_q$ and $\alpha_q + \beta_q$ respectively (see \cite[Proposition 3.3]{MR2920881} for the existence of $f'_{\mathrm{stab}}$; the choice of ordering of $\bar\chi_i(\Frob_q)$ is used here). Let $\mathfrak{m}^{(1)}$ and $\mathfrak{m}^{(0)}$ denote the corresponding maximal ideals of $H^{(1)}_{Q_m,(k'_1,k'_2)}$ and $H^{(0)}_{Q_m,(k'_1,k'_2)}$ determined by  $f'_{\mathrm{stab}}$. Let $\mathfrak{m}'$ denote the maximal ideal of $H_{(k'_1,k'_2)}$ determined by $f'$. Let $\TT_{Q_m,(k'_1,k'_2)}$ and $\TT_{(k'_1,k'_2)}$ denote the localizations of $H^{(1)}_{Q_m,(k'_1,k'_2)}$ and $H_{(k'_1,k'_2)}$  with respect to $\mathfrak{m}^{(1)}$  and $\mathfrak{m}'$ respectively. We get a sequence of natural inclusions of free $\OO$-modules:
\begin{align} \label{eq:inclusionSqstab}
S^{\ord}_{(k'_1,k'_2)}(\Gamma_{\Iw_p}, \mathrm{O})_{\mathfrak{m}'} \hookrightarrow S^{\ord}_{(k'_1,k'_2)}(\Gamma_{\Iw_p,Q_m^{(0)}}, \mathrm{O})_{\mathfrak{m}^{(0)}} \hookrightarrow S^{\ord}_{(k'_1,k'_2)}(\Gamma_{\Iw_p,Q_m^{(1)}}, \mathrm{O})_{\mathfrak{m}^{(1)}}. 
\end{align}

Let 
\begin{align*}
M_{Q_m,(k'_1,k'_2)} &\coloneqq  \Hom_{\mathrm{O}}\left(S^{\ord}_{(k'_1,k'_2)}(\Gamma_{\Iw_p,Q_m^{(1)}}, \mathrm{O})_{\mathfrak{m}^{(1)}}, \mathrm{O}\right), \quad M^{(0)}_{Q_m,(k'_1,k'_2)} \coloneqq  \Hom_{\mathrm{O}}\left(S^{\ord}_{(k'_1,k'_2)}(\Gamma_{\Iw_p,Q_m^{(0)}}, \mathrm{O})_{\mathfrak{m}^{(0)}}, \mathrm{O}\right), \\  M_{(k'_1,k'_2)} &\coloneqq \Hom_{\mathrm{O}}\left(S^{\ord}_{(k'_1,k'_2)}(\Gamma_{\Iw_p}, \mathrm{O})_{\mathfrak{m}'}, \mathrm{O}\right). 
\end{align*}
For each $q \in Q_m$, let $\Delta_q$ denote the $p$-primary part of  $\mathrm{Kli}^{(0)}_q/\mathrm{Kli}^{(1)}_q \cong (\Z/q\Z)^\times$. Let $\Delta_{Q_m} = \prod_{q \in Q_m} \Delta_q$. There is a natural action of $\mathrm{O}[\Delta_{Q_m}]$ on $M_{Q_m,(k'_1,k'_2)}$. Let $I_{Q_m}$ denote the kernel of the augmentation map $\mathrm{O}[\Delta_{Q_m}] \rightarrow \mathrm{O}$. 
 The inclusions in equation (\ref{eq:inclusionSqstab}) induce natural maps on their $\mathrm{O}$-duals:
 \begin{align} \label{eq:fixedbydeltaq}
 \dfrac{M_{Q_m,(k'_1,k'_2)} }{I_{Q_m} M_{Q_m,(k'_1,k'_2)}} \rightarrow M^{(0)}_{Q_m,(k'_1,k'_2)}  \rightarrow M_{(k'_1,k'_2)}.
 \end{align}

      \begin{lemma}
        \label{lem:deflem}
      Let $Q_m$ be a finite set (possibly empty) of Taylor--Wiles primes.  Let $k'_1, k'_2$ be positive integers such that $k'_i \equiv k_i \pmod{p-1}$ for $i=1,2$ and $k'_1 \geq k'_2 \geq 3$. There exists a representation $\rho_{Q_m,(k'_1,k'_2)}~:~G_{\Q,S \cup Q_m} \to \GSp_4(\TT_{Q_m,(k'_1,k'_2)})$ lifting $\bar\rho$ such that
        \begin{align}\label{eq:traceequalitytaylorwiles}\mathrm{Trace}\left(\rho_{Q_m,(k'_1,k'_2)}(\Frob_\ell)\right) = T_{\ell,2},\quad \text{for all }\ell \notin Q_m \cup S,\end{align}
        and such that $\rho_{Q_m,(k'_1,k'_2)} \mod (\mathfrak{m}^{(1)})^n$ belongs to $\mathbf{D}_{Q_m,(k'_1,k'_2)}\left(\dfrac{\TT_{Q_m,(k'_1,k'_2)}}{(\mathfrak{m}^{(1)})^n}\right)$ for all $n \geq 1$.
       
            \end{lemma}
        \begin{proof}

        The construction of the Galois representation $\rho_{Q_m,(k'_1,k'_2)}$ lifting $\bar\rho$ satisfying the trace equality given in equation (\ref{eq:traceequalitytaylorwiles}) and the first three properties listed in Definition \ref{defi:Qdef} of the functor $\mathbf{D}_{Q_m,(k'_1,k'_2)}$ follows from similar arguments given in the proof of Theorem \ref{thm:symplecticbasis}. The fourth property listed in Definition \ref{defi:Qdef} for Taylor--Wiles primes follows from work of Genestier--Tilouine \cite[Lemma 5.1.1, Section 9.3]{MR2234862}, see also Pilloni \cite[Proposition 6.5]{MR2920881}. In particular, if we consider a Taylor--Wiles prime $q$ in $Q_m$, then  $\rho_{Q_m,(k'_1,k'_2)}|_{G_{\Q_q}} \simeq \oplus_{i=1}^{4} \eta_{i,q}$, such that $\eta_{i,q} : G_{\Q_q} \to \TT_{Q_m,(k'_1,k'_2)}^{\times}$ is a character lifting $\bar\chi_{i,q}$ for every $1 \leq i \leq 4$, with $\eta_{1,q}$ and $\eta_{2,q}$ being unramified characters and $\eta_{2,q}\eta_{3,q}  = \eta_{1,q}\eta_{4,q}$. We also have 
        \begin{align}\label{eq:qinimageofRTdef}
        q U_{q,1} = \eta_{1,q}(\Frob_q) \eta_{2,q}(\Frob_q), \qquad U_{q,2} =  \eta_{1,q}(\Frob_q)+\eta_{2,q}(\Frob_q).
        \end{align}

  To conclude that the $\eta_{i,q}$'s take value in the ring $\TT_{Q_m,(k'_1,k'_2)}$, note that we have implicitly used the fact that our  residue field $\mathbb{F}$ is large enough so that Hensel's lemma ensures us that the coefficient ring contains all roots of Hecke polynomials of Taylor--Wiles primes. 
        \end{proof}    

Lemma \ref{lem:deflem} and equation (\ref{eq:qinimageofRTdef}) of Lemma \ref{lem:deflem} allows us to obtain surjective ring homomorphisms: 
\begin{align} \label{eq:RTTWsurj}
R^{\min}_{Q_m,(k'_1,k'_2)} \twoheadrightarrow \TT_{Q_m,(k'_1,k'_2)}, \qquad R^{\min}_{(k'_1,k'_2)} \twoheadrightarrow \TT_{(k'_1,k'_2)}. 
\end{align}
The surjective ring homomorphisms in equation (\ref{eq:RTTWsurj}) allow us to view $M_{Q_m,(k'_1,k'_2)}$ and $M_{(k'_1,k'_2)}$ as $R^{\min}_{Q_m,(k'_1,k'_2)}$ and $R^{\min}_{(k'_1,k'_2)}$ modules respectively. We have the following proposition:

\begin{proposition}[Genestier--Tilouine, Pilloni]\label{prop:TWsystemexistence}
For sufficiently large weights $k'_1 > k'_2 \gg 0$ such that $k'_i \equiv k_i \pmod{p-1}$ for $i=1,2$, the data \[\{R^{\min}_{(k'_1,k'_2)},\ \TT_{(k'_1,k'_2)}, \ M_{(k'_1,k'_2)},\ R^{\min}_{Q_m,(k'_1,k'_2)},\ \TT_{Q_m,(k'_1,k'_2)}, \ M_{Q_m,(k'_1,k'_2)}\}\] forms a Taylor--Wiles system (in the sense of \cite[Section 4.2.1.]{MR2920881}). That is, 
\begin{enumerate}
\item\label{point1} $R^{\min}_{(k'_1,k'_2)}$ is a complete local $\mathrm{O}$-algebra. $R^{\min}_{Q_m,(k'_1,k'_2)}$ is a complete local $\mathrm{O}[\Delta_{Q_m}]$-algebra with $R^{\min}_{(k'_1,k'_2)} \cong R^{\min}_{Q_m,(k'_1,k'_2)}/I_{Q_m}$.
\item\label{point2} $M_{(k'_1,k'_2)}$ is an  $R^{\min}_{(k'_1,k'_2)}$-module, that is free as an $\mathrm{O}$-module. $M_{Q_m,(k'_1,k'_2)}$ is an $R^{\min}_{Q_m,(k'_1,k'_2)}$-module. Furthermore, $M_{Q_m,(k'_1,k'_2)}$ is  a free $\mathrm{O}[\Delta_{Q_m}]$-module with finite rank that is independent of $m$.
\item\label{point3} $\TT_{(k'_1,k'_2)}$ equals the image of $R^{\min}_{(k'_1,k'_2)}$ inside $\mathrm{End}_{\mathrm{O}}(M_{(k'_1,k'_2)})$. Furthermore, $\TT_{Q_m,(k'_1,k'_2)}$ equals the image of $R^{\min}_{Q_m,(k'_1,k'_2)}$ inside $\mathrm{End}_{\mathrm{O}}(M_{Q_m,(k'_1,k'_2)})$.
\end{enumerate}
 Here, the sets $Q_m$ of Taylor--Wiles primes appear in Proposition \ref{prop:pilloni}. In addition, the following statements hold: 
\begin{enumerate}[label=(\roman*)]
\item\label{roman1} For all $m$, the image of $\TT_{Q_m,(k'_1,k'_2)}$ inside $\mathrm{End}_{\mathrm{O}}\left( \dfrac{M_{Q_m,(k'_1,k'_2)} }{I_{Q_m} M_{Q_m,(k'_1,k'_2)}}\right)$ is isomorphic to $\TT_{(k'_1,k'_2)}$.
\item\label{roman2} For all $m$, we have  $\dfrac{M_{Q_m,(k'_1,k'_2)} }{I_{Q_m} M_{Q_m,(k'_1,k'_2)}} \cong M_{(k'_1,k'_2)} $.
\end{enumerate}
\end{proposition}

\begin{proof}
Let $q \in Q_m$. We note that the restriction to $I_q$,  of the universal Galois deformation valued in $R^{\min}_{Q_m,(k'_1,k'_2)}$ is of the form $\begin{pmatrix} 1 & 0 & 0 & 0 \\ 0 & 1 &  0 & 0 \\ 0 & 0 & \chi_q & 0 \\ 0 & 0 & 0 & \chi_q\end{pmatrix}$ under a suitable basis. Here,  $\chi_q$ is a $R^{\min}_{Q_m,(k'_1,k'_2)}$-valued character that lifts the trivial character of $I_q$ and that is obtained from the restriction of a $R^{\min}_{Q_m,(k'_1,k'_2)}$-valued character of the decomposition group $G_{\Q_q}$. Hence, the character $\chi_q$ factors as $\chi_q: \Delta_q \rightarrow \left(R^{\min}_{Q_m,(k'_1,k'_2)}\right)^\times$. This allows us to obtain the $\mathrm{O}[\Delta_{Q_m}]$-structure on $R^{\min}_{Q_m,(k'_1,k'_2)}$. Since the universal deformation ring $R^{\min}_{(k'_1,k'_2)}$ parametrizes deformations of the type defined in Definition \ref{defi:Qdef}, that in particular are unramified at all the Taylor--Wiles primes $q \in Q_m$, the universal property gives us a natural isomorphism $R^{\min}_{(k'_1,k'_2)} \cong R^{\min}_{Q_m,(k'_1,k'_2)}/I_{Q_m}$. This proves point (\ref{point1}). 

The surjection in equation (\ref{eq:RTTWsurj}) provides us an $\mathrm{O}[\Delta_{Q_m}]$-structure on $M_{Q_m,(k'_1,k'_2)}$.  Using \cite[Theorems 2.2.5, 8.2.1]{MR2234862} (see also \cite[Theorems 3.5(2),3.6]{MR2920881}), one observes that, for each $q \in Q_m$, the action of $\Delta_q$ on $M_{Q_m,(k'_1,k'_2)}$ via the $p$-primary part of  $\mathrm{Kli}^{(0)}_q/\mathrm{Kli}^{(1)}_q \cong (\Z/q\Z)^\times$ matches with the $\Delta_q$-action obtained from equation (\ref{eq:RTTWsurj}). Pilloni's horizontal control theorem \cite[Theorem 9.1]{MR3059119} tells us that $M_{Q_m,(k'_1,k'_2)}$ is a free $\mathrm{O}[\Delta_{Q_m}]$-module and we have a natural isomorphism given by the first map of equation (\ref{eq:fixedbydeltaq}):
 \begin{align} \label{eq:isofirstmap}
 \dfrac{M_{Q_m,(k'_1,k'_2)} }{I_{Q_m} M_{Q_m,(k'_1,k'_2)}} \xrightarrow{\cong} M^{(0)}_{Q_m,(k'_1,k'_2)}.
 \end{align}
 To conclude that the second map of equation  (\ref{eq:fixedbydeltaq})
 \begin{align} \label{eq:isosecondmap}
 M^{(0)}_{Q_m,(k'_1,k'_2)}  \xrightarrow{\cong} M_{(k'_1,k'_2)}
 \end{align}
  is an isomorphism, one shows (following Genestier--Tilouine and Pilloni \cite{MR2234862, MR2920881}) that the homomorphisms
 \begin{align} \label{eq:levellowering}
S^{\ord}_{(k'_1,k'_2)}(\Gamma_{\Iw_p}, \mathrm{O})_{\mathfrak{m}'} \otimes_{\mathrm{O}} \mathrm{Frac}(\mathrm{O}) \rightarrow S^{\ord}_{(k'_1,k'_2)}(\Gamma_{\Iw_p,Q_m^{(0)}},\mathrm{O}) _{\mathfrak{m}^{(0)}} \otimes_{\mathrm{O}} \mathrm{Frac}(\mathrm{O}),
 \end{align}
 \begin{align} \label{eq:inversetoqstab}
 S^{\ord}_{(k'_1,k'_2)}(\Gamma_{\Iw_p}, \mathrm{O})_{\mathfrak{m}'} \otimes_{\mathrm{O}} \mathrm{Frac}(\mathrm{O})/\mathrm{O} \rightarrow S^{\ord}_{(k'_1,k'_2)}(\Gamma_{\Iw_p,Q_m^{(0)}},\mathrm{O})_{\mathfrak{m}^{(0)}}\otimes_{\mathrm{O}}\mathrm{Frac}(\mathrm{O})/\mathrm{O},
 \end{align}
 are isomorphisms. The fact that the map in equation (\ref{eq:levellowering}) is an isomorphism is a level lowering type argument. For example, see \cite[Proposition 6.2]{MR2920881}. The fact that the map in equation (\ref{eq:inversetoqstab}) is an isomorphism follows by constructing an explicit inverse to the $q$-stabilization map (as given in  \cite[Proposition 3.3]{MR2920881}, \cite[Corollary 3.2.4]{MR2234862}) which in turn relies on our choice of the ordering of the characters $\bar\chi_{i,q}$'s for Taylor--Wiles prime $q$ as given in Definition \ref{defn:primesTW}. See \cite[Proposition 6.3]{MR2920881}.  Combining isomorphisms in equations (\ref{eq:isofirstmap}) and (\ref{eq:isosecondmap}) gives us point (\ref{point2}) and point \ref{roman2}. 
 
 Point \ref{point3} follows directly from the surjections of equation (\ref{eq:RTTWsurj}). We have the following diagram:
 \begin{align*}
 \xymatrix{\TT_{Q_m,(k'_1,k'_2)} \ar@{.>}[d]\ar[r]& \mathrm{End}_{\mathrm{O}}\left(M_{Q_m,(k'_1,k'_2)} \right) \ar[r]& \mathrm{End}_{\mathrm{O}}\left( \dfrac{M_{Q_m,(k'_1,k'_2)} }{I_{Q_m} M_{Q_m,(k'_1,k'_2)}}\right) \ar[d]^{\cong} \\
\TT_{(k'_1,k'_2)} \ar[rr]& &  \mathrm{End}_{\mathrm{O}}\left(M_{(k'_1,k'_2)}\right)}
 \end{align*}
To see that one gets an induced surjective map $\TT_{Q_m,(k'_1,k'_2)} \rightarrow \TT_{(k'_1,k'_2)}$, one combines the construction of the Galois representation as given in Lemma \ref{lem:deflem}, the surjections of equation (\ref{eq:RTTWsurj}) and the Chebotarev density theorem (to see that traces of Frobenii of  primes outside any finite set containing $S \cup Q_m$ generate the Hecke algebras).  See \cite[Corollary 6.1]{MR2920881}. This allows us to obtain point \ref{roman1}. This completes the proof. 
\end{proof}

Propositions \ref{prop:pilloni} and \ref{prop:TWsystemexistence} allow us to apply \cite[Theorem 4.2]{MR2920881} to obtain the following ring isomorphism for $k'_1 > k'_2 \gg 0$, with $(k'_1,k'_2) \equiv (a_0,b_0) \pmod{(p-1)\Z^2}$. We note that \cite[Theorem 4.2]{MR2920881} is a modification of the Taylor--Wiles method as developed by Diamond \cite{MR1440309} and Fujiwara \cite{fujiwara2006deformation}. 
\begin{align} \label{eq:TWisofixedweightsk1k2}
  R^{\min}_{(k'_1,k'_2)} \cong  \TT_{(k'_1,k'_2)}.
\end{align}

\subsection{Proof of Theorem \ref{thm::rist}}

From Theorem \ref{thm:symplecticbasis}, we see that the Galois representation $\rho_\TT$ satisfies the deformation properties listed in Definition \ref{def:orddef}. This induces a ring homomorphism \[\Psi : R^{\min} \to \TT.\] Let $\rho^{\min} : G_{\Q,S} \to \GSp_4(R^{\min})$ denote the universal minimal ordinary deformation of $\bar\rho$. We note that   there exists a matrix $ P_{\rho^{\min},p} \in \GSp_4(R^{\min})$ such that for all $g \in G_{\Q_p}$, we have
    \[\rho^{\min}(g) = P_{\rho^{\min},p} \begin{pmatrix} \chi^{\min}_1(g) & * & * & * \\ 0 & \chi^{\min}_2(g) & * & *\\0 & 0 & \chi^{\min}_3(g) & * \\ 0 & 0 & 0 & \chi^{\min}_4(g)\end{pmatrix}P_{\rho^{\min},p}^{-1}.\]
          Here, for each $i \in \{1,2,3,4\}$, we have a character $\chi^{\min}_i:G_{\Q_p} \rightarrow (R^{\min})^\times$.  Furthermore, the character $\chi^{\min}_4$ is unramified. We let $\eta_i$ denote the restrictions of the characters $\chi_i^{\min}$ to the inertia group $I_p$. Local class field theory would tell us that these are characters of $\Z_p^\times$. Note that $\eta_2$ and $\eta_3$ lift $\omega^{a_0-1}$ and $\omega^{b_0-2}$ respectively. The restriction of the characters $\eta_2 \otimes \epsilon, \eta_3 \otimes \epsilon^2 : I_p \to (R^{\min})^\times$ to the pro-$p$ subgroup $1+p\Z_p$ of $\Z_p^\times$ induces ring homomorphisms $\psi_1 : \Z_p\llbracket X \rrbracket \to R^{\min}$ and $\psi_2 : \Z_p\llbracket Y \rrbracket \to R^{\min}$. Therefore, we get a map from $\Lambda \simeq \Z_p\llbracket X \rrbracket \widehat{\otimes} \Z_p\llbracket Y \rrbracket$ to $R^{\min}$ making it into a $\Lambda$-algebra (see \cite[Section 5.5]{MR2920881}). Under this $\Lambda$-structure, one notes that $\Psi$ is a $\Lambda$-algebra morphism.

To show that $\Psi$ is surjective, we note that the characteristic polynomial of $\Psi\circ\rho^{\min}(\Frob_\ell)$ (which equals $\rho_\TT(\Frob_\ell)$) matches with the Hecke polynomial at $\ell$, for all primes $\ell \notin S$. This would show that the image of $\Psi$ contains the Hecke operators $T_{\ell,2}, T_{\ell,1}$ and $T_{\ell,0}$  for all $\ell \notin S$. Theorem \ref{thm:symplecticbasis} tells us  \[\Psi \circ \chi^{\min}_4(\Frob_p) = U_{p,2}, \qquad \Psi \circ \dfrac{\chi^{\min}_3}{\omega^{b_0}\widetilde{\kappa}_2 \epsilon^{-2}}(\Frob_p) =U_{p,1}/U_{p,2}.\]
Here, $\widetilde{\kappa}_2:G_{\Q_p} \twoheadrightarrow 1+p\Z_p \hookrightarrow \Z_p\llbracket Y \rrbracket^\times \rightarrow (R^{\min})^\times$ is the tautological character obtained from $\psi_2$. Observe that this is indeed the character introduced in Theorem \ref{thm:symplecticbasis}. These arguments show that $\Psi$ is surjective. 

Suppose $(k'_1,k'_2) \equiv (a_0,b_0) \pmod{(p-1)\Z^2}$. Let $P_{(k'_1,k'_2)}$ denote the kernel of the ring homomorphism $\varphi_{(k'_1,k'_2)}:\Lambda \rightarrow \Z_p$. Universal properties of $R^{\min}$ and $R^{\min}_{(k'_1,k'_2)}$ provide us a natural isomorphism 
\begin{align}\label{iso:univdefspec}
\dfrac{R^{\min}}{P_{(k'_1,k'_2)}R^{\min}} \cong R^{\min}_{(k'_1,k'_2)}. 
\end{align}
When $k'_1 > k'_2 \gg 0$, equation (\ref{eq:TWisofixedweightsk1k2}) provides us an isomorphism $ R^{\min}_{(k'_1,k'_2)} \cong  \TT_{(k'_1,k'_2)}$. We get an induced commutative diagram where all morphisms are surjective. 
\begin{align*}
\xymatrix{
R^{\min}  \ar@{->>}[d]^{\otimes \Lambda/P_{(k'_1,k'_2)}}  \ar@{->>}[r]^{\Psi}& \TT  \ar@{->>}[d]^{\otimes \Lambda/P_{(k'_1,k'_2)}} \\
\dfrac{R^{\min}}{P_{(k'_1,k'_2)}R^{\min}} \ar[rd]^{\cong}  \ar@{->>}[r]& \dfrac{\TT}{P_{(k'_1,k'_2)} \TT}  \ar@{->>}[d]\\ &  \TT_{(k'_1,k'_2)} 
}
\end{align*}
 As a result, we obtain that whenever $ R^{\min}_{(k'_1,k'_2)} \cong  \TT_{(k'_1,k'_2)}$ holds, we obtain a natural induced isomorphism
\begin{align}\label{eq:perfectcontroliso}
 \dfrac{\TT}{P_{(k'_1,k'_2)}\TT}  \cong \TT_{(k'_1,k'_2)}.
\end{align}

Choose the minimal natural number amongst $\mathrm{Rank}_{\Z_p} \TT_{(k'_1,k'_2)}$, as we vary over all $(k'_1,k'_2) \equiv (a_0,b_0) \pmod{(p-1)\Z^2}$ such that the isomorphism $ R^{\min}_{(k'_1,k'_2)} \cong  \TT_{(k'_1,k'_2)}$ holds. Call it $n$. Topological Nakayama lemma \cite[\href{https://stacks.math.columbia.edu/tag/031D}{Tag 031D}]{stacks-project} tells us that $R^{\min}$ is finitely generated as a $\Lambda$-module and that we have $\Lambda$-module surjections:
\begin{align} \label{eq:isoranknfg}
\Lambda^n \twoheadrightarrow R^{\min} \stackrel{\Psi}{\twoheadrightarrow} \TT.
\end{align}
Going modulo $P_{(k'_1,k'_2)}$ with $(k'_1,k'_2)$ as  specified above, the minimality of $n$ and equation (\ref{eq:perfectcontroliso}) now provide  isomorphisms: 
\begin{align} \label{eq:isoranknfg2}
\Z_p^n \cong \dfrac{R^{\min}}{P_{(k'_1,k'_2)}R^{\min}}\cong  \dfrac{\TT}{P_{(k'_1,k'_2)}\TT}.
\end{align}
Since the set of height two prime ideals $P_{(k'_1,k'_2)}$'s with $(k'_1,k'_2) \equiv (a_0,b_0) \pmod{(p-1)\Z^2}$ and $k'_1 > k'_2 \gg 0$ are dense in $\mathrm{Spec}(\Lambda)$ (see \cite[Lemma A1]{hsieh_yoshida}), a standard density argument now tells us that the maps in equation (\ref{eq:isoranknfg}) are in fact isomorphisms. This completes the proof. \qed

\section{Proof of Theorem \ref{thm:yoshidafamily}}\label{sec:proofyoshidafamily}

{\theoremtwo*} 

Note that the proof of Theorem \ref{thm:R=T} (see equation (\ref{eq:perfectcontroliso})) provides us the following ring isomorphisms for $k'_1 > k'_2 \gg 0$ with $k'_1 \equiv k_1 \pmod{(p-1)}$ and $k'_2 \equiv k_2 \pmod{(p-1)}$:
\begin{align*}
\mathcal{T} \otimes_{\Z_p\llbracket y_1,y_2\rrbracket} \dfrac{\Z_p\llbracket y_1,y_2\rrbracket}{\mathcal{P}_{(k'_1,k'_2)}} &\cong \Rmin \otimes_{\Z_p\llbracket y_1,y_2\rrbracket} \dfrac{\Z_p\llbracket y_1,y_2\rrbracket}{\mathcal{P}_{(k'_1,k'_2)}} \cong \Rmin_{(k'_1,k'_2)} \cong \TT_{(k'_1,k'_2)}.
\end{align*}
Here, $\mathcal{P}_{(k'_1,k'_2)}$ is the prime ideal $((1+y_1)-(1+p)^{k'_1},(1+y_2)-(1+p)^{k'_2})$ of $\Z_p\llbracket y_1,y_2\rrbracket$.  Since the Hecke algebra $\TT_{(k'_1,k'_2)}$ is $\Z_p$-torsion free, \cite[Corollary 5.4]{MR2055355} lets us conclude that the ring $\TT_{(k'_1,k'_2)}$ is reduced.

Let $\mathcal{P}$ denote the prime ideal of $\mathcal{T}$ lying over $\mathcal{P}_{(k'_1,k'_2)}$ corresponding to $f$. Let $P$ denote the  prime ideal of $\TT_{(k'_1,k'_2)}$ corresponding to $\mathcal{P}$.  The localization $\left(\TT_{(k'_1,k'_2)}\right)_{P}$, being a reduced Artinian local ring, is a field. Consequently, the localization $\left(\dfrac{\TT}{\mathcal{P}_{(k'_1,k'_2)}\TT}\right)_{\mathcal{P}}$, being isomorphic to $\left(\TT_{(k'_1,k'_2)}\right)_{P}$, is also a field. This observation lets us conclude that $\mathcal{P}_{(k'_1,k'_2)}\TT_{\mathcal{P}}$ is the maximal ideal of $\TT_{\mathcal{P}}$. Thus, the maximal ideal of the localization $\TT_{\mathcal{P}}$ (which has Krull dimension $2$) is generated by $2$ elements. These observations let us conclude that $\TT_{\mathcal{P}}$ is a regular local ring (and hence a domain). It thus follows that there exists a unique minimal prime $\eta$ of $\TT$ contained in $\mathcal{P}$ (that corresponds to the zero ideal of the domain $\TT_{\mathcal{P}}$). \qed

\bibliography{biblio}
\bibliographystyle{amsplain}
\end{document}